\documentclass[final]{siamltex}
\usepackage{amsmath}
\usepackage[dvips]{graphicx}
\usepackage{textcomp}
\usepackage{amsbsy}
\usepackage{latexsym}
\usepackage[mathscr]{eucal}
\usepackage{amsfonts}
\usepackage{amssymb}
\usepackage{array}
\usepackage[usenames]{color}
\usepackage{showkeys}
\usepackage{xspace}

\usepackage{fullpage}

\newcommand{\bnew}[1]{{\color{black}{#1}}}

\begin{document}
\bibliographystyle{plain}

\title
{
Adaptive Finite Element Methods for Elliptic \\ 
Problems with Discontinuous Coefficients
}

\author{
Andrea Bonito\thanks{
Texas A\&M University, Department of Mathematics, TAMU 3368,
College Station, TX 77843, USA
({\tt bonito@math.tamu.edu}). 
}
\and
Ronald A. DeVore\thanks{
Texas A\&M University, Department of Mathematics, TAMU 3368,
College Station, TX 77843, USA
({\tt rdevore@math.tamu.edu}). 
}
\and
Ricardo H. Nochetto\thanks{
University of Maryland, Department of Mathematics and Institute for
Physical Science and Technology, College Park,
MD 20742, USA
({\tt rhn@math.umd.edu}). 
}
}

\newtheorem{prop}{Proposition}
\newtheorem{cor}{Corollary}
\newtheorem{remark}{Remark}
\newtheorem{example}{Example}
\newtheorem{proper}{Properties}
\newtheorem{assumption}{Assumption}
%
\def\RR{\rm \hbox{I\kern-.2em\hbox{R}}}
\def\NN{\rm \hbox{I\kern-.2em\hbox{N}}}
\def\ZZ{\rm {{\rm Z}\kern-.28em{\rm Z}}}
\def\CC{\rm \hbox{C\kern -.5em {\raise .32ex \hbox{$\scriptscriptstyle
|$}}\kern
-.22em{\raise .6ex \hbox{$\scriptscriptstyle |$}}\kern .4em}}
\def\vp{\varphi}
\def\<{\langle}
\def\>{\rangle}
\def\t{\tilde}
\def\i{\infty}
\def\e{\varepsilon}
\def\sm{\setminus}
\def\nl{\newline}
\def\o{\overline}
\def\wt{\widetilde}
\def\wh{\widehat}
\def\fT{{\mathfrak{T}}}
\def\cT{{\cal T}}
\def\cA{{\cal A}}
\def\cI{{\cal I}}
\def\cV{{\cal V}}
\def\cB{{\cal B}}
\def\cR{{\cal R}}
\def\cD{{\cal D}}
\def\cP{{\cal P}}
\def\cJ{{\cal J}}
\def\cM{{\cal M}}
\def\cO{{\cal O}}
\def \mod {\rm mod}
\def\Chi{\raise .3ex
\hbox{\large $\chi$}} \def\vp{\varphi}
\def\lsima{\hbox{\kern -.6em\raisebox{-1ex}{$~\stackrel{\textstyle<}{\sim}~$}}\kern -.4em}
\def\lsim{\hbox{\kern -.2em\raisebox{-1ex}{$~\stackrel{\textstyle<}{\sim}~$}}\kern -.2em}
\def\[{\Bigl [}
\def\]{\Bigr ]}
\def\({\Bigl (}
\def\){\Bigr )}
\def\[{\Bigl [}
\def\]{\Bigr ]}
\def\({\Bigl (}
\def\){\Bigr )}
\def\L{\pounds}
\def\pr{{\rm Prob}}
\newcommand{\cs}[1]{{\color{magenta}{#1}}}
\def\ds{\displaystyle}
\def\ev#1{\vec{#1}}     
\newcommand{\lt}{\ell^{2}(\nabla)}
\def\Supp#1{{\rm supp\,}{#1}}
\def\R{\mathbb{R}}
\def\E{\mathbb{E}}
\def\nl{\newline}
\def\T{{\relax\ifmmode I\!\!\hspace{-1pt}T\else$I\!\!\hspace{-1pt}T$\fi}}
\def\N{\mathbb{N}}
\def\Z{\mathbb{Z}}
\def\N{\mathbb{N}}
\def\Zd{\Z^d}
\def\Q{\mathbb{Q}}
\def\C{\mathbb{C}}
\def\Rd{\R^d}
\def\gsim{\mathrel{\raisebox{-4pt}{$\stackrel{\textstyle>}{\sim}$}}}
\def\sime{\raisebox{0ex}{$~\stackrel{\textstyle\sim}{=}~$}}
\def\lsim{\raisebox{-1ex}{$~\stackrel{\textstyle<}{\sim}~$}}
\def\div{\mbox{ div }}
\def\M{M}  \def\NN{N}                  
\def\L{{\ell}}               
\def\Le{{\ell^1}}            
\def\Lz{{\ell^2}}
\def\Let{{\tilde\ell^1}}     
\def\Lzt{{\tilde\ell^2}}
\def\Ltw{\ell^\tau^w(\nabla)}
\def\t#1{\tilde{#1}}
\def\la{\lambda}
\def\La{\Lambda}
\def\ga{\gamma}
\def\BV{{\rm BV}}
\def\Ga{\eta}
\def\al{\alpha}
\def\cZ{{\cal Z}}
\def\cA{{\cal A}}
\def\cU{{\cal U}}
\def\argmin{\mathop{\rm argmin}}
\def\argmax{\mathop{\rm argmax}}
\def\prob{\mathop{\rm prob}}

\def\cO{{\cal O}}
\def\cA{{\cal A}}
\def\cC{{\cal C}}
\def\cF{{\cal F}}
\def\bu{{\bf u}}
\def\bz{{\bf z}}
\def\bZ{{\bf Z}}
\def\bI{{\bf I}}
\def\cE{{\cal E}}
\def\cD{{\cal D}}
\def\cG{{\cal G}}
\def\cI{{\cal I}}
\def\cJ{{\cal J}}
\def\cM{{\cal M}}
\def\cN{{\cal N}}
\def\cT{{\cal T}}
\def\cU{{\cal U}}
\def\cV{{\cal V}}
\def\cW{{\cal W}}
\def\cL{{\cal L}}
\def\cB{{\cal B}}
\def\cG{{\cal G}}
\def\cK{{\cal K}}
\def\cS{{\cal S}}
\def\cP{{\cal P}}
\def\cQ{{\cal Q}}
\def\cR{{\cal R}}
\def\cU{{\cal U}}
\def\bL{{\bf L}}
\def\bl{{\bf l}}
\def\bK{{\bf K}}
\def\bC{{\bf C}}
\def\X{X\in\{L,R\}}
\def\ph{{\varphi}}
\def\D{{\Delta}}
\def\H{{\cal H}}
\def\bM{{\bf M}}
\def\bx{{\bf x}}
\def\bj{{\bf j}}
\def\bG{{\bf G}}
\def\bP{{\bf P}}
\def\bW{{\bf W}}
\def\bT{{\bf T}}
\def\bV{{\bf V}}
\def\bv{{\bf v}}
\def\bt{{\bf t}}
\def\bz{{\bf z}}
\def\bw{{\bf w}}
\def \span{{\rm span}}
\def \meas {{\rm meas}}
\def\rhom{{\rho^m}}
\def\lll{\langle}
\def\argmin{\mathop{\rm argmin}}
\def\argmax{\mathop{\rm argmax}}

\newcommand{\bcT}{{ \bolfcal T}}
\def\dJ{\nabla}
\newcommand{\ba}{{\bf a}}
\newcommand{\bb}{{\bf b}}
\newcommand{\bc}{{\bf c}}
\newcommand{\bd}{{\bf d}}
\newcommand{\bs}{{\bf s}}
\newcommand{\bff}{{\bf f}}
\newcommand{\bp}{{\bf p}}
\newcommand{\bg}{{\bf g}}
\newcommand{\by}{{\bf y}}
\newcommand{\br}{{\bf r}}
\newcommand{\be}{\begin{equation}}
\newcommand{\ee}{\end{equation}}
\newcommand{\bea}{$$ \begin{array}{lll}}
\newcommand{\eea}{\end{array} $$}
\def \Vol{\mathop{\rm  Vol}}
\def\fT{{\mathfrak{T}}}
\def \mes{\mathop{\rm mes}}
\def \Prob{\mathop{\rm  Prob}}
\def \exp{\mathop{\rm    exp}}
\def \sign{\mathop{\rm   sign}}
\def \sp{\mathop{\rm   span}}
\def \vphi{{\varphi}}
\def \csp{\overline \mathop{\rm   span}}
\newcommand{\KL}{Karh\'unen-Lo\`eve }
%
\newcommand{\beqn}{\begin{equation}}
\newcommand{\eeqn}{\end{equation}}

\newenvironment{Proof}{\noindent{\bf Proof:}\quad}{\end{proof}}

\renewcommand{\theequation}{\thesection.\arabic{equation}}
\renewcommand{\thefigure}{\thesection.\arabic{figure}}

\makeatletter
\@addtoreset{equation}{section}
\makeatother

\newcommand\abs[1]{\left|#1\right|}
\newcommand\clos{\mathop{\rm clos}\nolimits}
\newcommand\trunc{\mathop{\rm trunc}\nolimits}
\renewcommand\d{d}
\newcommand\dd{d}
\renewcommand\diag{\mathop{\rm diag}}
\newcommand\dist{\mathop{\rm dist}}
\newcommand\diam{\mathop{\rm diam}}
\newcommand\cond{\mathop{\rm cond}\nolimits}
\newcommand\eref[1]{{\rm (\ref{#1})}}
\newcommand{\iref}[1]{{\rm (\ref{#1})}}
\newcommand\Hnorm[1]{\norm{#1}_{H^s([0,1])}}
\def\int{\intop\limits}
\renewcommand\labelenumi{(\roman{enumi})}
\newcommand\lnorm[1]{\norm{#1}_{\ell^2(\Z)}}
\newcommand\Lnorm[1]{\norm{#1}_{L_2([0,1])}}
\newcommand\LR{{L_2(\R)}}
\newcommand\LRnorm[1]{\norm{#1}_\LR}
\newcommand\Matrix[2]{\hphantom{#1}_#2#1}
\newcommand\norm[1]{\left\|#1\right\|}
\newcommand\ogauss[1]{\left\lceil#1\right\rceil}
\newcommand{\QED}{\hfill
\raisebox{-2pt}{\rule{5.6pt}{8pt}\rule{4pt}{0pt}}%
  \smallskip\par}
\newcommand\Rscalar[1]{\scalar{#1}_\R}
\newcommand\scalar[1]{\left(#1\right)}
\newcommand\Scalar[1]{\scalar{#1}_{[0,1]}}
\newcommand\Span{\mathop{\rm span}}
\renewcommand\supp{\mathop{\rm supp}}
\newcommand\ugauss[1]{\left\lfloor#1\right\rfloor}
\newcommand\with{\, : \,}
\newcommand\Null{{\bf 0}}
\newcommand\bA{{\bf A}}
\newcommand\bB{{\bf B}}
\newcommand\bR{{\bf R}}
\newcommand\bD{{\bf D}}
\newcommand\bE{{\bf E}}
\newcommand\bF{{\bf F}}
\newcommand\bH{{\bf H}}
\newcommand\bU{{\bf U}}
\newcommand\cH{{\cal H}}
\newcommand\sinc{{\rm sinc}}
\def\enorm#1{| \! | \! | #1 | \! | \! |}

\newcommand{\dm}{\frac{d-1}{d}}

\let\bm\bf
\newcommand{\bbeta}{{\mbox{\boldmath$\beta$}}}
\newcommand{\bal}{{\mbox{\boldmath$\alpha$}}}
\newcommand{\bbi}{{\bm i}}

\def\nnew{\color{Red}}
\def\mnew{\color{Blue}}

\newcommand{\dI}{\Delta}
%
%


\def\Vb{\mathbb V}
\def\V{\mathbb V}
\def\P{\mathbb P}
\def\Sb{\mathbb S}
\def\na{\nabla}
\def\bq{\mathbf q}

\newcommand{\what}{\widehat}
\newcommand{\veps}{\varepsilon}

\def\eps{\veps}
\def\GREEDY{\textsf{GREEDY}\xspace}
\def\SOLVE{\textsf{SOLVE}\xspace}
\def\MARK{\textsf{MARK}\xspace}
\def\REFINE{\textsf{REFINE}\xspace}
\def\ESTIMATE{\textsf{ESTIMATE}\xspace}
\def\MAKECON{\textsf{MAKECONFORM}\xspace}
\def\REDUCEOSC{\textsf{ADAPTDATA}\xspace}
\def\UPDATE{\textsf{UPDATE}\xspace}
\def\OSC{\textsf{DATA}\xspace}
\def\MARKDATA{\textsf{MARK\_DATA}\xspace}
\def\AFEM{\textsf{AFEM}\xspace}
\def\MAIN{\textsf{PDE}\xspace}
\def\COEFF{\textsf{COEFF}\xspace}
\def\COEFFQ{\textsf{COEFFQ}\xspace}
\def\APPROX{\textsf{RHS}\xspace}
\def\DISC{\textsf{DISC}\xspace}
\def\GAL{\textsf{GAL}\xspace}
\def\pA{q}
\def\CONDP{{\bf Condition~$\bp$ }}
\def\CONF{\textsf{CONF}\xspace}
\newcommand{\osc}{\textrm{osc}}
\newcommand{\As}[1][s]{\ensuremath{\mathbb{A}_{#1}}} 
\newenvironment{algotab}%
{\par\begin{samepage}%
\begin{tabbing}\ttfamily%
 \hspace*{5mm}\=\hspace{3ex}\=\hspace{3ex}\=\hspace{3ex}\=\hspace{3ex}%
\=\hspace{3ex}\=\hspace{3ex}\=\hspace{3ex}\=\hspace{3ex}\kill}%
{\end{tabbing}\end{samepage}}

\maketitle
\date{}

 \begin{abstract}
 Elliptic partial differential equations (PDEs) with discontinuous diffusion
 coefficients occur in application domains such as diffusions  through porous
 media, electro-magnetic field propagation on heterogeneous media, and diffusion
 processes on rough surfaces.
 The standard approach to numerically treating such problems using finite element methods is to assume that the discontinuities
 lie on the boundaries of the cells in the  initial triangulation.  However, this does not match applications where discontinuities occur on curves, surfaces,
 or manifolds, and could even be unknown  beforehand.  One of the obstacles to treating such  discontinuity problems is that the usual perturbation theory for elliptic PDEs assumes bounds for
 the distortion of the coefficients in the $L_\infty$ norm and this in turn requires that  the discontinuities are matched exactly when the coefficients are approximated.
 We present a new approach based on distortion of the coefficients in an $L_q$ norm with $q<\infty$ which  therefore does not require the exact matching of the 
discontinuities.   We then use this new distortion theory
 to formulate new adaptive finite element methods (AFEMs) for such 
 discontinuity problems.
 We show that such AFEMs are optimal in the sense of
distortion versus number of computations,  and report insightful 
numerical results supporting our analysis.
\end{abstract}



 \section{Introduction}   We consider elliptic partial differential equations of the following form
\begin{equation}
 \begin{aligned}
 \label{problem}
 -{\rm div}(A\nabla u)&=f,\quad {\rm on}\ \Omega\\
 u&= 0,\quad {\rm on}\  \partial \Omega.
 \end{aligned}
\end{equation}
 where $\Omega$ is a polyhedral  domain in $\R^d$, $d\geq 1$ integer, and $A=(a_{ij})_{i,j=1}^d$ is a $d\times d$ positive definite matrix of $L_\infty(\Omega)$ functions.  
 
  We let  $|\cdot |$ denote the Euclidean norm on $\R^d$ and when  $w:\Omega\to  \R^d$ is a vector valued function defined on $\Omega$ then we set
 \be
 \label{nc1}
 \|w\|_{L_p(\Omega)}:=\| ~ |w| ~ \|_{L_p(\Omega)},
 \ee
 for each $0<p\leq \infty$.
Similarly, if $B$ is any $d\times d$ matrix, then  $\|B\|$ denotes its spectral norm (its norm as an operator from $\ell_2(\R^d)$ to itself).
If $B$ is a matrix valued function on $\Omega$ then we define the norms
 \be
 \label{nc2}
 \|B\|_{L_p(\Omega)}:=\| ~ \|B\| ~ \|_{L_p(\Omega)}.
 \ee
 By redefining the $a_{ij}$  on a set of measure zero, we  may assume that each $a_{ij}$ is defined everywhere on $\Omega$ and  
 \be
 \label{boundedA}
 \|A(x)\|\le \|A\|_{L_\infty(\Omega)}, \quad x\in\Omega.
 \ee

  As usual, we  interpret \eref{problem}
 in the  weak sense and use the Lax-Milgram theory for existence and uniqueness.  Accordingly, we let $H_0^1(\Omega)$ be the Sobolev space  of real valued functions on $\Omega$ which vanish on the boundary of $\Omega$ equipped with  the norm
 \be
 \label{Hnorm}
 \|v\|_{H_0^1(\Omega)}:=\|\nabla v\|_{L_2(\Omega)}
 \ee
and we define the quadratic form
 \be
 \label{qf}
 a(u,v):=\int_\Omega (A \nabla u)\cdot \nabla v,\quad u,v \in H_0^1(\Omega).
 \ee
 Throughout, we shall use   $a\cdot b$ to denote the   inner product of  vectors $a$ and $b$.
 
 To ensure uniform ellipticity, we assume that $A$ is symmetric and uniformly positive definite a.e.
 on $\Omega$.  Again, without loss of generality, we can redefine $A$ on a set of measure zero so
 that $A(x)$ is uniformly positive definite everywhere on $\Omega$.      Given a positive definite, symmetric matrix $B$, we  denote by $ \lambda_{\min}(B)$ its   smallest eigenvalue and by $\lambda_{\max}(B)$ its largest eigenvalue.
 In the case that $B$ is a function of $x\in\Omega$, we define
$$
 \lambda_{\min}(B):=\inf_{x\in \Omega}\lambda_{\min}(B(x)),
 $$
 and
 $$  \lambda_{\max}(B):=\sup_{x\in \Omega}\lambda_{\max}(B(x))=\|\lambda_{\max}(B(\cdot))\|_{L_\infty(\Omega)}=\|B\|_{L_\infty(\Omega)}.
 $$
 It follows that
 \be
 \label{pd}
 \lambda_{\min} (A)|y|^2\le y^t A(x) y\le \lambda_{\max}(A) |y|^2 ,\quad  \forall x\in \Omega, \ y\in \R^d.
 \ee
 Let us also note that \eref{pd} implies  
 \be
 \label{coercive}
 \lambda_{\min}(A)\|v\|^2_{H_0^1(\Omega)}\le a(v,v)\le  \lambda_{\max}(A)\|v\|^2_{H_0^1(\Omega)},
 \ee
 for all $v\in H_0^1(\Omega)$.  That is, the energy norm induced by 
$a(\cdot,\cdot)$ is equivalent to the $H_0^1$ norm.
       
 Given $f\in H^{-1}(\Omega):=H_0^1(\Omega)^*$  (the dual space
   of $H_0^1(\Omega)$), the Lax-Milgram theory implies the existence of a
   unique $u=u_f \in H_0^1(\Omega)$ such that
   \be
   \label{wproblem}
   a(u,v)=\langle f,v\rangle,\quad v\in H_0^1(\Omega),
   \ee
   where $\langle f,v\rangle$ is the $H^{-1}-H_0^1 $ dual pairing.
    
     Practical numerical algorithms for solving \eref{wproblem}, i.e. finding an
    approximation to $u$ in $H_0^1(\Omega)$ to any prescribed accuracy
    $\veps$,  begin by approximating $f$ by an $\hat f$ and $A$ by
    an $\hat A$; this is the case, for example, when quadrature rules are applied.  
    To analyze the performance of such an algorithm
    therefore requires an estimate for the effect of such a
    replacement.  The usual form of such a perturbation result is the
    following (see e.g. \cite{LarssonThomee}).  Suppose that both $A,\hat A$ are
    symmetric, positive definite and satisfy 
    \begin{equation}
      \label{eigenbound} 
     r\le \lambda_{\min }(A),\lambda_{\max}(A)\le M,  \qquad\hat r\le \lambda_{\min }(\hat A),\lambda_{\max}(\hat A)\le \hat M,    
     \end{equation}
    for some $0<r\le M<\infty$ and $0<\hat r \le \hat M < \infty$.  Then,

    \begin{equation}\label{perturbation1}
\|u  - \hat u\|_{H_0^1(\Omega)} \leq 
 \hat r^{-1} \Big(\| f-\hat f\|_{H^{-1}(\Omega)} 
+ r^{-1}\|A-\hat A\|_{L_\infty(\Omega)}\|f\|_{H^{-1}(\Omega)}\Big),
\end{equation}
where $\hat u \in H^1_0(\Omega)$ is the solution of \eqref{wproblem} with
diffusion matrix $\hat A$ and right hand side $\hat f$. 
 If $A$ has discontinuities, then for  \eref{perturbation1} to be useful,  the approximation $\hat A$ would have to match these discontinuities in order for the right
 side to be small.  In many applications, the discontinuities of $A$ are either unknown or lie along curves and surfaces
 which cannot be captured exactly.  This precludes the direct use of \eref{perturbation1} in the construction and analysis of numerical methods for \eref{problem}.  
 
  The first  goal of the present paper is to describe a perturbation theory, given in Theorem \ref{T:perturbation} of \S\ref{T:perturbation}, which replaces \eref{perturbation1} by the bound
   \be 
   \label{goal1}
\|u-\hat u\|_{H_0^1(\Omega)}\le \hat r ^{-1}
\Big(\|f-\hat f\|_{H^{-1}(\Omega)}+  \|\nabla u\|_{L_{p}(\Omega)} \|A-\hat A\|_{L_q(\Omega)}\Big), \quad q:=\frac{2p}{p-2}\ee
provided $\nabla u\in L_p(\Omega)$ for some $p\ge 2$. Notice that when
$p=2$ this estimate is of the same form as   \eref{perturbation1}
because $\|\nabla u\|_{L_2(\Omega)}\le r^{-1}\|f\|_{H^{-1}(\Omega)}$.
The advantage of \eref {goal1} over 
\eref{perturbation1} is that we do not have to match the discontinuities of $A$
exactly for the right side to be small.
Note however that we still require bounds on the eigenvalues of $A$ and in
particular $A \in L^\infty(\Omega)$.

However, estimate \eqref{goal1} exhibits an asymmetry in the
dependency of the eigenvalues of $A$ and $\hat A$ and requires additional
assumptions on the right side $f$ to guarantee that $\nabla u \in L_p(\Omega)$.
This issue is discussed in \S\ref{ss:Lp}.  It turns out that there is a
range of $p>2$, depending only on $\Omega$ and the constants $r,M$ such that $f\in
W^{-1}(L_p(\Omega))$ (the dual of $W^1_{\bnew{0}}(L_{\frac{p}{p-1}}(\Omega))$) implies $\nabla u\in L_p(\Omega)$ and so the estimate
\eref{goal1} can be applied for such $f$.  The restriction that $f\in W^{-1}(L_p(\Omega))$, for some $p>2$, is quite mild and is met by all applications that we envisage.

 The second goal of this paper, is to develop an adaptive finite element method (AFEM)
  applicable to \eref{wproblem} \bnew{primarily} when $A$ possesses
 discontinuities not aligned with the meshes and thus not resolved by
 the finite element approximation in $L_\infty$.
 \bnew{Although piecewise polynomial approximation of $A$ beyond
   piecewise constant is unnecessary
   for the foremost example of discontinuous diffusion coefficients
   across a Lipschitz co-dimension one manifold, we
   emphasize that our theory and algorithm apply to any polynomial
   degree. Higher order approximations of $A$ may indeed be relevant in
   dealing with $A$'s with point discontinuities (see Section 5 in \cite{Me63}) or 
   $A$'s which are piecewise smooth.}

 We develop \bnew{AFEM} based on newest vertex
  bisection in \S\ref{s:AFEM} and prove that  our  method has a certain
 optimality in terms of rates of convergence.  We note that it is convenient  to
 restrict our discussion to   newest vertex subdivision and the case $d=2$ for
 notational reasons.  
However, all of our results hold for more general $d\ge 2$ and 
other refinement procedures  such as those discussed in \cite{BN:10}.
 
  The adaptive algorithm that we propose and analyze is based on three subroutines \APPROX, \COEFF, and \MAIN.  The first of these gives an approximation to $f$ using piecewise polynomials.
  This type of  approximation of $f$ is quite standard in AFEMs.  The subroutine \COEFF  produces an
   approximation $\hat A$ to $A$ in $L_q$.   We need,  however,
 that $\hat A$ is uniformly positive definite with bounds on the
 eigenvalues of $\hat A$ comparable to the bounds assumed on $A$, 
a restriction that seems on the surface to be in conflict with approximation
 in $L_q$. \bnew{The only exception is piecewise
   constant $\hat A$'s because then $\hat A$ can be taken to be the meanvalue
   of $A$ elementwise for all $q\ge2$; see \S \ref{s:numerics}.}
 We show in \S\ref{SapproxA} that on a theoretical level the
 restriction of positive definiteness of $\hat A$  
 does not effect the approximation order in $L_q$.  However, the derivation of
 numerically implementable algorithms which  ensure positive definiteness and
 perform optimally in terms of $L_q(\Omega)$ approximation is a more subtle
 issue because there is a need to clarify in what sense $A$ is provided to us.  
 We leave this aspect as an open area for further study.
Finally, we denote by \MAIN  the standard AFEM method \cite{NSV:09}, but
based on the approximate right hand side  $\hat f$
and diffusion coefficient $\hat A$ provided by \APPROX and \COEFF.
 
We end this paper by providing two insightful numerical 
experiments on the performance of the new algorithm along with the key fact
that \eqref{goal1} can be applied locally.
  
 \section{ Perturbation Argument}\label{S:perturbation}
 \label{S:Perturbation}
 In this section, we prove a perturbation theorem which allows for the
 approximation of $A$ to take place  in a norm weaker than $L_\infty$.
As we shall see, this in turn requires $\nabla u \in L_p(\Omega)$ for some
$p>2$.
Validity of such bounds is discussed in \S\ref{ss:Lp}.

\subsection{The Perturbation Theorem}\label{SS:multidimensional}

 Let $A, \hat A\in [L_\infty(\Omega)]^{d\times d}$ be symmetric, positive
definite matrices satisfying \eqref{eigenbound},
for some $r, \hat r>0$ and some $M, \hat M < \infty$, and let
$f, \hat f\in H^{-1}(\Omega)$.
Let $u, \hat u \in H^1_0(\Omega)$ be the solution of
\eqref{wproblem} and of the perturbed problem
\begin{equation}
\label{pproblem}
\int_\Omega (\hat A\nabla \hat u) \cdot \nabla v = \langle \hat f, v \rangle , \qquad \forall
v \in H^1_0(\Omega).
\end{equation}
 We now prove that the map $A\mapsto u$ is Lipschitz continuous from
$L_q(\Omega)$ to $H^1_0(\Omega)$. This map is shown to be continuous
in \cite[\S 8, Theorem 3.1]{DelfourZolesio}.

\begin{theorem}[perturbation theorem]\label{T:perturbation}
 For any $p\ge 2$, the functions $u$ and $\hat u$ satisfy
 \be 
 \label{PT1}
\|u-\hat u\|_{H_0^1(\Omega)}\le  \hat r^{-1}\|f-\hat
f\|_{H^{-1}(\Omega)}+ {\hat r}^{-1} \|\nabla u\|_{L_{p}(\Omega)} \|A-\hat
A\|_{L_q(\Omega)},\quad q:=\frac{2p}{p-2} \in [2,\infty]\ee
provided $\nabla u\in L_p(\Omega)$.
\end{theorem}
\begin{proof}  Let $\bar u$ be the solution to \eref{problem} with diffusion matrix $\hat A$ and right side $f$.  Then, from the perturbation estimate \eref{perturbation1}, we have
 \be
 \label{TP1}
 \|\hat u-\bar u\|_{H_0^1(\Omega)}\le \hat r^{-1}\|f-\hat f\|_{H^{-1}(\Omega)}.
 \ee
We are therefore left with bounding $\|u-\bar u\|_{H_0^{1}(\Omega)}$.  From the definition of $u$ and $\bar u$, we have
$$
\int_\Omega  (A \nabla u) \cdot \nabla v = 
\int_\Omega  (\hat A  \nabla  \bar u)\cdot \nabla v,
$$
for all $v\in H_0^1(\Omega)$.   This gives
$$
\int_\Omega [\hat A\nabla( u-\bar u)]\cdot \nabla v=\int_\Omega [(\hat A-A)\nabla u]\cdot \nabla v,
$$
for all $v\in H_0^1(\Omega)$.  Taking $v=u-\bar u$, we obtain
$$
\int_\Omega  [\hat A\nabla(u-\bar u) ]\cdot \nabla( u- \bar u) = \int_\Omega
[(\hat A-A) \nabla u]\cdot  \nabla (u-\bar u)\le 
\bnew{\|(\hat A-A) \nabla u \|_{L_2(\Omega)} \|\nabla (u-\bar u)\|_{L_2(\Omega)}.}
$$
\bnew{If we use the coercivity estimate \eref{coercive} with $A$ replaced
by $\hat A$,  then we deduce
$$
\hat r\|u-\bar u\|_{H_0^1(\Omega)} \le \|(A-\hat A)\nabla u\|_{L_2(\Omega)}.
$$ 
Applying H\"older inequality to the right side with $p\ge 2$ and
$q=2p/(p-2)$ we arrive at
\begin{equation}\label{Tm2}
\|u-\bar u\|_{H_0^1(\Omega)}\le  {\hat r}^{-1} \|\nabla u\|_{L_{p}(\Omega)} \|A-\hat A\|_{L_{q}(\Omega)}.
\end{equation}
}
Combining this with \eref{TP1}, we infer that
\begin{eqnarray*}
\|u-\hat u\|_{H_0^1(\Omega)}&\le&  \|u-\bar u\|_{H_0^1(\Omega)}+ \| \bar u-\hat
u\|_{H_0^1(\Omega)}\\
&\le & \hat r^{-1} \|\nabla u\|_{L_{p}(\Omega)} \|A-\hat A\|_{L_{q}(\Omega)} +\hat r ^{-1}\|f-\hat f\|_{H^{-1}(\Omega)},\nonumber
\end{eqnarray*}
as desired.  
\end{proof}

\begin{remark}[local perturbation estimates]\label{r:local}
\rm We point out that the choice of $p$ in the perturbation estimate \eqref{PT1}
could be different from one subdomain of $\Omega$ to another.
To fix ideas, assume that $\Omega$ is decomposed  into two subdomains $\Omega_1$
and $\Omega_2$.
Similar arguments as provided in the previous lemma yield
$$
\| u- \hat u\|_{H^1_0(\Omega)} \leq \hat r^{-1}\|f-\hat f\|_{H^{-1}(\Omega)}+
\hat r^{-1} \| A-\hat A\|_{L_{q_1}(\Omega_1)}  \| \nabla
u \|_{L_{p_1}(\Omega_1)} + \hat r^{-1} \| A-\hat A\|_{L_{q_2}(\Omega_2)}  \| \nabla
u \|_{L_{p_2}(\Omega_2)},
$$
where $p_i \in [2,\infty]$ and $q_i=2p_i/(p_i-2)$, $i=1,2$.
As we shall see in \S\ref{s:numerics}, this turns out to be critical
when the jump in the coefficients takes place in a subdomain
 $\Omega_i$ with the solution
$u\in W^1_\infty(\Omega_i)$, thereby allowing to take $p_i=\infty$.
\end{remark}
 
 \subsection{Sufficient conditions for $\nabla u$ to be in
   $L_p$}\label{ss:Lp}
 In order for Theorem \ref{T:perturbation} to be relevant we need that $\nabla u$ is in $L_p$ for some
 $p>2$.   It is therefore of interest to know of sufficient conditions on $A$ and the right side $f$ for this to be the case. In this section, we shall recall some known results in this direction.  
 
 From the Lax-Milgram theory, we know that  the solution operator boundedly maps
 $H^{-1}(\Omega)$ into $H^1_0(\Omega)$.   It is natural to ask whether this mapping property extends to $p>2$, that is,  whether we have 
 \nl 
 \noindent
 \CONDP: {\it For each $f\in W^{-1}(L_p(\Omega))$, the solution $u=u_f$ satisfies
 \be
 \label{condp}
|u|_{W^1(L_p(\Omega))}:= ||\nabla u||_{L_p(\Omega)} \leq C_p \| f \|_{W^{-1}(L_p(\Omega))},
\ee
with the constant $C_p$ independent of $f$.}

\begin{remark}[local \CONDP]
\rm As already noted in Remark \ref{r:local}, it is not necessary for the $p$ to be
uniform over $\Omega$. 
In particular, one could decompose $\Omega$ on subdomains on which \CONDP is
valid for different $p$'s.
This is used in \S\ref{s:numerics} for the numerical illustration of the method.
\end{remark}
 
 When $A=I$ (the case of  Laplace's equation),  the validity of \CONDP is
 a well studied problem in Harmonic Analysis.   It is known that  for each
 Lipschitz domain $\Omega$, there is a $P>2$ which depends on $\Omega$ such that
 \CONDP holds for all $2\le p\le P$ (see for example  Jerison and
 Kenig \cite{JK95}).   
In fact, one have in this setting
$P>4$ when $d=2$ and $P>3$ when $d=3$.  
For later use  when $A=I$, 
we denote by $K$ the constant   depending only  on $\Omega$ and $P$
for which
\begin{equation}\label{e:Lp_Id}
|| \nabla u ||_{L_P(\Omega)} \leq K || f ||_{W^{-1}(L_P(\Omega))}.
\end{equation}

For more general $A$, \CONDP  can be shown to hold by using   a perturbation argument given by Meyers
\cite{Me63}(see also Brenner and Scott \cite{BS08}).   We shall describe Meyers' result only in the case $p>2$. We let 
\be
\label{defeta}
 \eta(p) := \frac{1/2 - 1/p}{1/2- 1/P},
 \ee
  and note that $\eta(p)$ increases from the value zero at $p=2$ to
  the value one at $p=P$. 
For any $t\in(0,1)$, we define
\begin{equation}\label{d:p*}
p^*(t):=\arg\max \{ K^{-\eta(p)}>1-t: 2<p<P\}. 
\end{equation}
With these definitions in hand, we have the following result for
general $A$.  Although this result is known \bnew{(see Meyers
  \cite{Me63})}, we provide the 
following simple proof for completeness of this section.

\begin{prop}[\bnew{membership in $W^1(L_p(\Omega)$}]\label{p:meyers}
Assume that $f$ and $\Omega$ are such that for some $P>2$ and some
constant $K$, the
solution $u \in H^1_0(\Omega)$ of problem \eref{wproblem} for 
Laplace's equation satisfies  \eqref{e:Lp_Id}
whenever  $f \in W^{-1}(L_P(\Omega))$.  
 If \eref{eigenbound} is valid for $A$, then 
the solution $u\in H^1_0(\Omega)$ of \eref{wproblem} satisfies
$$
||\nabla u||_{L_p(\Omega)} \leq C \| f \|_{W^{-1}(L_p(\Omega))},
$$
provided  $2\le p<p^*(r/M)$ and $C:= \frac 1 M\frac{K^{\eta(p)}}{1- K^{\eta(p)} (1-\frac{r}{M})}$.
\end{prop}
\begin{proof}
The main idea of the proof is to write $A$ as a perturbation of the
identity and deduce the $L^p$-bound on $\nabla u$ from the
$L^p$-bound for the solution of the Poisson problem.

The operator $T:= -\Delta$ is invertible from 
$H^{-1}(\Omega)$ to $H^1_0(\Omega)$,
and its inverse $T^{-1}$ is bounded with norm one.  From  \eref{e:Lp_Id}, it is 
also bounded with norm $K$ as a mapping from $W^{-1}(L_P(\Omega))$ to $W^1_0(L_P(\Omega))$, where we define the  norm on  $W^1_0(L_P(\Omega)) $ by its  semi-norm.
For  the real method of interpolation, we have for $2<p<P$, $W_0^1(L_p(\Omega))  =
\lbrack  H^1_0(\Omega),W^1_0(L_P(\Omega)) 
\rbrack_{\eta(p), p}$,
where $\eta(p)$ is defined in \eref{defeta}.
It follows  by interpolation that $T^{-1}$ is a bounded mapping from $W^{-1}(L_p(\Omega))$ to $W_0^1(L_p(\Omega))$ and  
$$
||\nabla T^{-1}f ||_{L^p(\Omega)} \leq K^{\eta(p)}||f||_{W^{-1}(L_p(\Omega))}.
$$

Let $S:W_0^{1}(L_p(\Omega))\rightarrow W^{-1}(L_p(\Omega))$ denote the operator satisfying
$
  Sv := -\div \big(\frac 1 M A\nabla v \big).
$
For convenience, we also define   the perturbation operator $Q:=T-S$.
Then,
$S$ and $Q$ are bounded operators  from $W_0^{1}(L_p(\Omega))$ to $W^{-1}(L_p(\Omega))$ with norms
$$
\| S \| \leq 1\qquad \text{and} \qquad \| Q \| \leq 1-\frac{r}{M}.
$$
It follows that as a mapping from $W_0^{1}(L_p(\Omega))$ to $W_0^{1}(L_p(\Omega))$
$$
\| T^{-1} Q \| \leq \| T^{-1} \| \| Q \| \leq K^{\eta(p)} (1-\frac{r}{M}) .
$$
Hence, 
$S=T(I-T^{-1}Q)$ is  invertible provided $K^{\eta(p)} (1-\frac{r}{M})<1$, that is, provided $2\le p<p^*(r/M)$.
Moreover, as a mapping from $W^{-1}(L_p(\Omega))$ to $W_0^{1}(L_p(\Omega))$
$$
\|S^{-1}\|\le \frac{\|T^{-1}\|}{1- K^{\eta(p)} (1-\frac{r}{M})}\le  \frac{K^{\eta(p)}}{1- K^{\eta(p)} (1-\frac{r}{M})},
$$
which yields the desired bound.
\end{proof}

\section{Adaptive Finite Element Methods}\label{s:AFEM}
There is by now a  considerable literature which constructs and analyzes AFEMs.   Our  new algorithm differs from those existing in the literature
in the assumptions we make on the diffusion matrix $A$.  Typically, it is assumed that each entry in this  matrix is a piecewise polynomial on the initial partition $\cT_0$
or at a minimum that it is piecewise smooth on the partition $\cT_0$. 
\bnew{{\it Our algorithm does not require the assumption that the
discontinuities of $A$ are compatible with $\cT_0$ or even known to us a
priori, except for the knowledge of the Lebesgue exponent $p$ of 
$\|\nabla u\|_{L_p(\Omega)}$ or equivalently $q=2p/(p-2)$. 
However, the universal choice $q=2$
is valid for the practically significant case of piecewise constant
$A$ over subdomains separated by a
Lipschitz manifold of co-dimension one;  \bnew{see \S \ref{s:numerics}}}}.
Our algorithms use subroutines that appear in the standard AFEMs
and can be seen as an extension of \cite{Ste7,CDN} where the approximation of
$f$ is discussed. 
Therefore, we shall review the existing algorithms in this section.  We refer the reader to Nochetto et al. \cite{NSV:09} for an
up to date survey of the current theory of AFEMs for elliptic problems.   Unless noted otherwise, the proofs of all the results quoted here can be found in  \cite{NSV:09}.

\subsection{Partitions and Finite Element Spaces}\label{SS:partitions}

Underlying any AFEM is a method for adaptively  partitioning the domain into
polyhedral cells.     Since there are, by now, several papers which give a
complete presentation of refinement rules used in AFEMs, for example
\cite{BN:10}, we assume the reader is familiar with these methods of
partitioning.  
In the discussion that follows, we will  consider 
the two dimensional case  (triangles) and the method of newest
vertex bisection,  but the results we present hold  for   $d\ge 2$ and
more general refinement rules  satisfying Conditions 3, 4 and 6 in
\cite{BN:10}. In particular, they hold for successive bisections, 
quad-refinement, and red-refinement all with {\it hanging nodes.} 
It is simply for notational convenience that we limit our discussion
 to newest vertex bisection.

The starting point for newest vertex partitioning is to  assume that $\Omega$ is a polygonal domain and $\cT_0$ is an
  initial partition  of $\Omega$ into a finite number of triangles
 each with a newest vertex label. It
  is assumed that the initial labeling of vertices of
  $\cT_0$ is {\it compatible}; see \cite{04BDD,Ste06}.  
If a cell is to be refined, it is divided into two cells by bisecting 
 the edge opposite to the newest vertex and labeling the newly created vertex
for the two children cells.
This bisection rule gives a unique refinement procedure and
an ensuing forest $\fT$ emanating from the root $\cT_0$.
  
 We say a partition $\cT\in\fT$ is {\it admissible} if it can be obtained from
  $\cT_0$ by a finite number of newest vertex bisections.  The complexity of
  $\cT$ can be measured by the number $n(\cT)$ of bisections that need to be
  performed to obtain $\cT$ from   $\cT_0$: in fact, $\#\cT=\#\cT_0+n(\cT)$.
  We denote by  $\fT_n$, $n\ge 1$,    the set of all   partitions $\cT$ that can
be obtained from $\cT_0$ by $n$ newest vertex bisections.

  A general triangulation $\cT\in\fT_n$ may be non-conforming, i.e.,  contain hanging nodes. If
  $\cT$ is non-conforming, then it is known \cite{04BDD,Ste06,BN:10} that it can be refined to a conforming
  partition $\overline{\cT}$ by applying a number of newest vertex
  bisections  controlled by $n(\cT)$, namely,
  \be
  \label{nv1}
   \#\overline{\cT}-\#\cT_0\le C_0n(\cT),
   \ee
    with $C_0$ an absolute constant depending only on the initial partition
    $\cT_0$ and its labeling. We denote by 
    $$\CONF(\cT)$$
     the smallest conforming admissible partition which contains $\cT$.
      
Given a conforming partition $\cT\in\fT_n$ and a polynomial
degree $m_u\ge1$, we define $\V(\cT)$ to be
the finite element space of continuous piecewise polynomials of degree at
most $m_u$ subordinate to $\cT$. Given a positive definite diffusion
matrix $A \in L_\infty(\Omega)$, and a right side $f\in L_2(\Omega)$,   
the Galerkin approximation $U:=U(\cT,A,f):=\GAL(\cT,A,f)$
of \eqref{wproblem} is by definition the unique solution of the discrete problem
\be
\label{gal}
U\in \mathbb V(\cT): \qquad
\int_\Omega (A\nabla U)\cdot \nabla V = \int_\Omega f \ V, \qquad \forall V\in
\mathbb V(\cT).
\ee
  
Notice that given $\cT$, the function $U$ is the best
  approximation to $u$ from $\V(\cT)$ in the energy norm induced by $A$ which is in turn equivalent to the $H_0^1(\Omega)$ norm.

\subsection {The structure of AFEM}\label{SS:AFEM}
 Standard AFEMs for approximating $u$ generate a sequence of nested 
 admissible, conforming partitions 
$\{\cT_k\}_{k\geq 0}$ of $\Omega$ starting from $\cT_0$.    The partition $\cT_{k+1}$ is obtained from    
$\cT_k$, $k\ge 0$,   by  using  an 
adaptive strategy.  
\bnew{Given any partition $\cT$ and finite element function $V\in
  \mathbb V(\cT)$, the \bnew{residual estimator} is defined as
\begin{equation*}
\begin{aligned}
\eta_\cT(V,A,f;\cT)&:=\left(\sum_{T\in \cT} \eta_\cT(V,A,f;T)^2\right)^{1/2}, \\
\eta_\cT(V,A,f;T)&:=\text{diam}(T)\| f+\text{div}(A\nabla V)\|_{L_2(T)} + \left(\sum_{F \in \Sigma(T)} \diam(F) \| \lbrack A\nabla V\rbrack \|_{L_2(F)}^2\right)^{1/2},
\end{aligned}
\end{equation*}
where $\Sigma(T)$ is the set of edges (d=2) or faces (d=3) constituting the boundary of $T$ and $\lbrack \cdot \rbrack$ denotes the normal jump across $F$.
The accuracy of the Galerkin solution $U_k =\GAL(\cT_k,A,f) \in
\mathbb V(\cT_k)$ is asserted by examining $\eta_{\cT_k}(U_k,A,f;T)$
and \bnew{{\it marking}
certain cells in $\cT_k$ for refinement} via a D\"orfler marking \cite{Dor96}.
After performing these refinements (and possibly additional
refinements to remove hanging nodes), we obtain a new
\bnew{conforming} partition. 
This process is repeated until the \bnew{residual estimator} is below a prescribed tolerance $\veps_k$.
The corresponding subdivision is declared to be $\cT_{k+1}$ and its associated Galerkin solution \bnew{$U_{k+1}\in\V(\cT_{k+1})$.}
In the case where $A$ and $f$ are piecewise polynomials subordinate to
$\cT_{k+1}$, we recall that the \bnew{residual estimator} is
equivalent to the energy error, i.e. there exists constants $C_L \leq
C_U$ only depending on the shape regularity of \bnew{the forest $\fT$} and on the
eigenvalues of $A$ such that
\begin{equation}\label{e:equiv}
C_L \eta_{\cT_{k+1}}(U_{k+1},A,f;\cT_{k+1}) \leq \| u - U_{k+1} \|_{H^1_0(\Omega)} \leq C_U \eta_{\cT_{k+1}}(U_{k+1},A,f;\cT_{k+1}). 
\end{equation}
\bnew{Instrumental to our arguments} is the absence of so-called
oscillation terms \cite{MNS:00,08CKNS,NSV:09} in the above relation, which
follows from considering piecewise polynomial $A$ and $f$; 
\bnew{we refer to \cite{NSV:09,Ste7}.}
 
 We denote this procedure by \MAIN and formally write %
$$
[\cT_{k+1},U_{k+1}] = \MAIN  (\cT_k,A,f,\veps_k),
\qquad  \eta_{\cT_{k+1}}(U_{k+1},\bnew{A,f,}\cT_{k+1})  \le \veps_{k}.
$$
In other words the input to \MAIN is the partition $\cT_k$, the matrix $A$, the right side $f$ and the target error $\veps_k$.
The output is the partition $\cT_{k+1}$ and the new Galerkin solution $U_{k+1}$ which satisfies the error bound
\begin{equation}\label{upper-eps}
\| u -U_{k+1} \|_{H^1_0(\Omega)} \leq C_U \veps_k.
\end{equation}
}
\bnew{Each loop within \MAIN is a contraction for the energy error
  with a constant $\alpha<1$ depending on $C_L,C_U$ and the marking parameter
\cite{NSV:09}. Therefore, if $\hat\veps_k :=
\eta_{\cT_k}(U_k,A,f;\cT_k)$ is the level of error before the call to \MAIN, 
then the number of iterations $i_k$ within \MAIN to reduce such an error to
$\veps_k$ is bounded by 
\bnew{
\begin{equation}\label{inner-iter}
i_k \le \frac{\log\left(\frac{C_U M^{1/2}}{C_L r^{1/2}}\right) + 
\log \left(\frac{\hat\veps_k}{\veps_k}\right)}{\log\alpha^{-1}} + 1.
\end{equation}
}}

This idealized algorithm does not carefully handle the error
incurred in the formulation and solution of \eqref{gal}, namely in the
procedure $\GAL(\cT,A,f)$ \cite{NSV:09}. 
This step requires the computation of integrals that are products of $f$
or $A$ with functions from the finite element space.  In performance analysis of such algorithms, it is 
typically assumed that these integrals are  computed exactly, while in fact they are computed by quadrature rules.  The effect of quadrature is
not  assessed in a pure a posteriori context. 
One
alternative, advocated in \cite{04BDD,Ste7} for the Laplace operator, 
is to approximate $f$ by a suitable piecewise polynomial $f_k$ over $\cT_k$. Of
course, one still needs to understand in what sense $f$ and $A$ are
given to us, a critical  issue not addressed here.

Our AFEM differs from $\MAIN$ in that we use approximations to both
$f$ and $A$, the latter being crucial to the method.  
\bnew{Given a current partition $\cT_k$ and a target
tolerance $\veps_k$,} the AFEM will
first find an admissible conforming partition $\cT_k'$, 
which is a refinement of $\cT_k$, 
on which we can approximate $f$  by a piecewise polynomial $f_k$ and
likewise $A$ by a piecewise polynomial $A_k$ such that
\be
\label{firststep}
 \|f-f_k\|_{H^{-1}(\Omega)}\le \veps_k',\quad 
\|A-A_k\|_{L_q(\Omega)}\le \veps_k',
\ee
with $q=2p/(p-2)\in[2,\infty]$ (the existing algorithms in the literature always take $q=\infty$
   \cite{NSV:09}).
The tolerance $\veps_k'$ is chosen as a multiple of $\veps_k$, for
example $\veps_k'=\omega\veps_k$ with $\omega >0$ yet to be determined.  
We next apply $[\cT_{k+1},U_{k+1}] = \MAIN (\cT_k',A_k,f_k,\veps_k/2)$
to find the new admissible conforming partition $\cT_{k+1}$, which is
a refinement of $\cT_k'$, and so of $\cT_k$, and Galerkin solution
\bnew{$U_{k+1}\in\V(\cT_{k+1})$} satisfying
$$
\|u_k-U_{k+1}\|_{H_0^1(\Omega)}\le 
\bnew{C_U \eta_{\cT_{k+1}}(U_{k+1},A_k,f_k;\cT_{k+1}) \le
\frac{C_U}{2}\veps_k,}
$$
where $u_k$ is the  solution to \eref{wproblem} with diffusion matrix
$A_k$ and right side $f_k$. From the perturbation estimate  \eref{PT1}, with
$\hat r>0$ a bound for the minimum eigenvalue of $A_k$, we obtain
\begin{equation*}
\begin{aligned}
\|u-U_{k+1}\|_{H_0^1(\Omega)}&\le
\|u-u_k\|_{H_0^1(\Omega)}+\|u_k-U_{k+1}\|_{H_0^1(\Omega)} 
\\
&\le \hat r^{-1}\|f-f_k\|_{H^{-1}(\Omega)}+
\hat r^{-1} \|\nabla
u\|_{L_p(\Omega)}\|A-A_k\|_{L_q(\Omega)}+ \bnew{\frac{C_U}{2}\veps_k.}
\end{aligned}
\end{equation*}
Therefore, invoking \eqref{condp} and choosing
$\veps_k'$ (or $\omega$) sufficiently small, we get the desired bound 
\begin{equation}\label{secondstep}
\|u-U_{k+1}\|_{H_0^1(\Omega)} \le
 \hat r^{-1} \big(1+C_p\|f\|_{W^{-1}(L_p\Omega))} \big)\veps_k' +
 \bnew{\frac{C_U}{2} \veps_k
\le C_U \veps_k.}
\end{equation}

We see that such an AFEM has three basic subroutines.  At
  iteration $k$, the first one is an
algorithm \APPROX which provides the approximation $f_k$ to $f$,  
the second one is an algorithm \COEFF which provides the
approximation $A_k$ to $A$, and the third one is \MAIN which does the marking
and further refinement $\cT_{k+1}$ to drive down the error of the Galerkin
approximation $U_{k+1}$ to $u$.  We discuss each of these in somewhat  more
detail now.

 We  denote by \APPROX the algorithm which generates the approximation to $f$.  It takes as input a function
 $f\in H^{-1}(\Omega)$, 
 a conforming partition $\cT$,  and  a tolerance $\veps$.  The
 algorithm then outputs 
$$[\hat \cT,\hat f] = {\APPROX}(f,\cT,\veps)$$
where
 $\hat \cT$ is a conforming partition which is a refinement of $\cT$ and $\hat
 f$ is a piecewise polynomial of degree at most $m_f$ subordinate to $\cT$ such that 
\be
\label{approx}
\|f-\hat f\|_{H^{-1}(\Omega)}\le \veps.
\ee
Notice that we do not assume any  regularity for $\hat f$.
In theory, one could construct such $\APPROX$, but in practice one
needs more information on $f$ to realize such algorithm as we now discuss.

By far, the majority of AFEMs assume that $f\in L_2(\Omega)$ but recent work
\cite{Ste7,CDN} treats the case of  certain more general right sides $f\in
H^{-1}(\Omega)$.  If $f\in L_2(\Omega)$, then one can bound the
error in approximating $f$ by piecewise polynomials of degree at most $m_f$ by
\be
\label{osc}
\|f-\hat f\|_{H^{-1}(\Omega)} \le \osc (f,\cT):=\(\sum_{T\in \cT}h_T^2 \|f-a_T(f)\|_{L^2(T)}^2\)^{\frac 1 2},
\ee
where $a_T(f)$ is the $L^2(T)$
orthogonal projection 
of $f$ onto \bnew{$\P_{m_f}(T)$}, the space of polynomials of total degree  
\bnew{$\le m_f$} over $T$, and $\hat f|_T := a_T(f)$ for $T \in \cT$.  
 The right side of \eref{osc} is called the oscillation of $f$ on
 $\cT$ \bnew{\cite{MNS:00,NSV:09}}.  
For any concrete  realization of {\APPROX}, one needs a model
for what information is available about $f$.
We refer to \cite{CDN} for further discussions in this direction.

Similarly, one needs to approximate $A$ in the AFEM. 
 Given a positive definite and bounded diffusion
matrix $A$, a conforming partition $\cT$ and a tolerance $\veps$,
the procedure
$$[\hat\cT,\hat A] = \COEFF (A,\cT,\veps)$$
outputs a conforming partition $\hat \cT$, which is a refinement of
$\cT$, and  a diffusion matrix $\hat A$, which has piecewise polynomial components of degree at most $m_A$ subordinate to $\hat \cT$ and satisfies 
   \be
   \label{COEFF}
   \|A-\hat A\|_{L_q(\Omega)}\le \veps,
   \ee
for $q=2p/(p-2)\in[2,\infty]$.
   In addition, in order  to guarantee the positive definiteness of $\hat A$, we
   require  that there is a known constant $C_2$ for which we have
   \be
   \label{COEFF1}
  C_2^{-1} \gamma_{\min}(A)\le \gamma_{\min}(\hat A)\le \gamma_{\max}(\hat A)\le C_2\gamma_{\max}(A).
  \ee
In \S \ref{SapproxA} we  discuss   constructions of $\COEFF$ (and
briefly mention constructions for  $\APPROX$) that have the
above properties and in addition are optimal in the sense of \S\ref{SS:Performance}.

If $A$ is a piecewise polynomial matrix of degree $\le m_A$
on the initial partition $\cT_0$,
then one would have an exact representation of $A$ as a polynomial on each cell $T$ of any partition $\cT$ and there is no need to approximate $A$.  For more general
$A$, the standard approach is to  approximate $A$  in the $L_\infty(\Omega)$ norm by piecewise
polynomials.   This requires that   $A$ is piecewise smooth on the initial
partition $\cT_0$ in order to guarantee that this $L_\infty$ error can be made
arbitrarily small; thus $q=\infty$.   The new perturbation theory we have given allows one to  circumvent this restrictive assumption on $A$ required by
standard AFEM. Namely, it is enough to  assume that \CONDP holds for some $p>2$ since then
$q<\infty$.

\subsection{Measuring the performance of AFEM}\label{SS:Performance}
The ultimate goal of an AFEM is to produce a {\it quasi-best} approximation $U$ to $u$ 
with error measured in $\|\cdot\|_{H_0^1(\Omega)}$. The performance of the AFEM
is measured by  the size of $\|u-U\|_{H_0^1(\Omega)}$ relative            
to the size of the partition  $\cT$. The size of $\cT$ usually
reflects the total computational cost of implementing the algorithm.
As a benchmark, it is useful to compare the performance of the AFEM
with  the best approximation of $u,f$ and $A$ 
provided we have full knowledge of them. 


\medskip\noindent
{\bf Approximating $u$.}
For each $n\ge 1$, we define 
\bnew{$\Sigma_n^{\bnew{m_u}}$}  to be the union of all the \bnew{finite element
  spaces} $\V(\cT)\subset H^1_0(\Omega)$
with $\cT\in\fT_n$ (the set of non-conforming
partitions obtained from $\cT_0$ by at most $n$ refinements). 
Notice that \bnew{$\Sigma_n^{\bnew{m_u}}$} is a nonlinear class of functions.  Given any function
   $v\in H_0^1(\Omega)$, we denote by 
   $$
   \bnew{\sigma_n(v)_{H_0^1(\Omega)}:=\sigma_n^{m_u}(v)_{H_0^1(\Omega)}}:=\inf_{\bnew{V\in\Sigma_n^\bnew{m_u}}}
   \|v-V\|_{H_0^1(\Omega)},\quad n\ge 1,
   $$
    the error of the best approximation of $v$ by elements of \bnew{$\Sigma_n^\bnew{m_u}$}.
   Using $\sigma_n$, we can stratify the space $H_0^1(\Omega)$ into 
   approximation classes: for any $s>0$,
   we define $\cA^s:=\cA^s(\cT_0,H_0^1(\Omega))$ as the set of all $v\in H_0^1(\Omega)$ for which
   \be
   \label{As}
  |v|_{\cA^s} := \sup_{n\ge 1} \Big( n^s\sigma_n(v)_{H_0^1(\Omega)}\Big)
  < \infty;
  \ee
    the quantity $|v|_{\cA^s}$ is a quasi-semi-norm.
   Notice that as $s$ increases the cost of membership to be in
   $\cA^s$ increases.  For example, we have $\cA^{s_1}\subset\cA^{s_2}$ whenever
   $s_2\le s_1$.   One can only expect \eqref{As} for a certain range
   of $s$, namely $0<s\le S$, where $S=m_u/d$ is the natural bound on 
   the order of approximation imposed by the polynomial degree $m_u$ being used.

  Since the output of AFEMs are conforming partitions, it is important to understand whether the imposition that the partitions are conforming
  has any serious effect on the approximation classes.  In view of \eref{nv1}, we have  that for any $v\in\cA^s$, there exist $S_n\in\Sigma_n^\bnew{m_u}$ subordinate to a conforming partition for which
  \be
  \label{confapprox}
  \|v-S_n\|_{H_0^1(\Omega)}\le C_0^s|v|_{\cA^s}n^{-s}, \quad n\ge 1,
  \ee
  with  $C_0$ the constant in \eref{nv1}.  Thus, if we had defined the approximation classes  $\cA^s$ with the additional requirement that the underlying partitions are conforming, then we would get the same approximation class and an equivalent  quasi-norm.

  As mentioned above, the input to routine \MAIN  includes polynomial approximations
  $\hat f$ and $\hat A$ to  $f$ and $A$ and then the  algorithm produces an approximation to  the solution $\hat
  u$ of \eqref{wproblem} with diffusion coefficient $\hat A$ and
  right hand side $\hat f$.
   However, $u \in \cA^s$ does not guarantee that $\hat u \in
    \cA^s$, which motivates us to introduce the following definition.

\begin{definition}[$\veps-$approximation of order $s$]
  Given $u\in \cA^s$ and $\veps>0$, a function $v$ is said to be an
  \emph{$\veps-$approximation of order $s$ to $u$} if 
  $\|u-v\|_{H^1_0(\Omega)}\le\veps$ and there
  exists a constant $C$ independent of $\veps$, $u$, and $v$, such that for 
  all $\delta\geq \veps$ there exists $n\in \mathbb N$ with
  \begin{equation}\label{only}
  \sigma_n(v)_{H_0^1(\Omega)} \le \delta,\quad n\le  C |u|_{\cA^s}^{1/s}\delta^{-1/s}.
  \end{equation}
\end{definition}

\noindent
We remark that if $\veps_1<\veps_2$ and $v$ is an
  $\veps_1-$approximation of order $s$ to $u$, then $v$ is an $\veps_2-$approximation
  of order $s$ to $u$ as well.
  We now provide a lemma characterizing such functions.
\begin{lemma}[$\veps-$approximations of order $s$]\label{l:approx}
Let $u \in \cA^s(\cT_0,H^1_0(\Omega))$ and $v \in H^1_0(\Omega)$
satisfy $\| u - v\|_{H^1_0(\Omega)}\leq \veps$ for some $\veps>0$.
Then $v$ is a $2\veps$-approximation of order $s$ to $u$.
\end{lemma}
\begin{proof}
Let $\delta\geq 2\veps$.
It suffices to invoke a triangle inequality to realize that 
\be
\label{mt3}
\sigma_n(v)_{H_0^1(\Omega)}\le \|u-v\|_{H_0^1(\Omega))}+\sigma_n(u)_{H_0^1(\Omega)}\le   \delta/2 +\sigma_n(u)_{H_0^1(\Omega)}.
\ee
 Since $u\in \cA^s(\cT_0,H^1_0(\Omega))$ we deduce that there exists
$n\leq |u|_{\cA^s}^{1/s} (\delta/2)^{-1/s}$ such that
$\sigma_n(u)_{H_0^1(\Omega)}\leq \delta/2$.
Estimate \eqref{only} thus follows with $C=2^{1/s}$.
\end{proof}

 In view of this discussion, we make the assumption that the call
$[\hat\cT,\hat U] = \MAIN(\cT,\hat A,\hat f,\veps)$ deals with approximate
data $\hat A$ and $\hat f$ exactly and creates no further errors. 
We say that $\MAIN$ is
of {\it class optimal performance in $\cA^s$} if there is an absolute
constant $C_3$ such that the number of
elements $N(\hat u)$ marked for refinement on $\cT$ to achieve the
error $\|\hat u -\hat U\|_{H^1_0(\Omega)}\le\veps$ satisfies
\be
\label{optmain}
N(\hat u) \le C_3 |\hat u|^{1/s}_{\cA^s} \veps^{-1/s},
\ee
whenever the solution $\hat u$  to  \eqref{wproblem} with data $\hat A$ and $\hat f$ is in $\cA^s(\cT_0,H^1_0(\Omega))$.
  This is a slight abuse of terminology  because this algorithm 
  is just {\it near class optimal} due to the presence of $C_3$.
  We drop the word 'near' in what follows for this and other
  algorithms.
  Moreover, notice that we distinguish between the elements selected 
  for refinement by the algorithm and those chosen to ensure
  conforming meshes. 
  Estimate \eqref{optmain} only concerns the former since the latter
  may not satisfy
  \eqref{optmain} in general; see for instance \cite{04BDD,NSV:09}.

\medskip\noindent
 {\bf Approximating $f$.}  We can measure the performance of 
 the approximation of $f$ in a similar way.  We let \bnew{$\Sigma_n^{m_f}$}
 be the space of all piecewise polynomials $S$ of degree at most $ m_f \geq 0$
 subordinate to a partition $\cT\in\fT_n$, and then define
  $$
   \sigma_n(f)_{H^{-1}(\Omega)}:=\sigma^{m_f}_n(f)_{H^{-1}(\Omega)}:=\inf_{S\in\Sigma_n^{m_f}}\|f-S\|_{H^{-1}(\Omega)},\quad n\ge 1.
  $$   %
    In analogy with the class $\cA^s$, we let the class 
   $\cB^s:=\cB^s(\cT_0,H^{-1}(\Omega))$, $s\ge 0$, consist of all functions $f\in H^{-1}(\Omega)$ for which 
   \be
   |f|_{\cB^s}:= \sup_{n\ge 1} \Big( n^s\sigma_n(f)_{H^{-1}(\Omega)}
   \Big) < \infty.
  \ee

 We will also need to consider the approximation of functions in other norms.  If $0< q\le \infty$ and $g\in L_q(\Omega)$ ($g\in C(\overline{\Omega})$ in the case $q=\infty$), we define
  $$
   \sigma_n(f)_{L_q(\Omega)}:=\sigma^{m_f}_n(f)_{L_q(\Omega)}:=\inf_{S\in\Sigma_n^{m_f}}\|f-S\|_{L_q(\Omega)},\quad n\ge 1,
  $$   %
    and the corresponding approximation classes 
   $\cB^s(L_q(\Omega)):=\cB^s(\cT_0,L_q(\Omega))$, $s\ge 0$, 
consisting  of all functions $L_q(\Omega)$ for which 
   \be
   |f|_{\cB^s(L_q(\Omega))}:= \sup_{n\ge 1} \Big( n^s\sigma_n(f)_{L_q(\Omega)}
   \Big) < \infty.
  \ee

   Let us now see what performance we can expect of the algorithm
   \APPROX.  If $f\in\cB^s(\cT_0,H^{-1}(\Omega))$, then there are 
   partitions $\cT^*$ with $\#\cT^*-\#\cT_0\le |f|^{1/s}_{\cB^s}\veps^{-1/s}$ on which we can find a piecewise polynomial $S$ such that $\|f-S\|_{H^{-1}(\Omega)}\le\veps$.
   Given any $\cT$ obtained as a refinement of $\cT_0$, the overlay $\cT^*\oplus
   \cT$ of $\cT^*$ with $\cT$  has cardinality obeying \cite{NSV:09}
\begin{equation}\label{e:overlay}
 \# (\cT^*\oplus \cT) - \#\cT \le \#\cT^*-\#\cT_0 
\leq |f|^{1/s}_{\cB^s}\veps^{-1/s}.
\end{equation}
Notice that at this stage the partitions might not be conforming.
This motivates us to say that  the algorithm \APPROX has {\it class optimal
performance} on $\cB^s$ if the number  $N(f)$ 
of elements chosen to be refined by the
algorithm to achieve a tolerance $\veps$ starting from $\cT$,  always satisfies
   \be
   \label{rhsapprox}
    N(f) \le C_3 |f|^{1/s}_{\cB^s}\veps^{-1/s},
   \ee
  with $C_3$ an absolute constant.   
  We refer to \S\ref{SapproxA} for the construction of such algorithms.

\medskip\noindent
 {\bf Approximating $A$.}
With slight abuse of notation, we denote again by \bnew{$\Sigma_n^{m_A}$} the class
of piecewise polynomial matrices of degree $\leq m_A$ subordinate to a 
partition $\cT\in\fT_n$. The best approximation error of $A$ within
$\Sigma_n^{m_A}$ is given by
$$
\sigma_n(A)_{L_q(\Omega)} := \sigma_n^{m_A}(A)_{L_q(\Omega)}
:= \inf_{S\in\Sigma_n^{m_A}} \|A-S\|_{L_q(\Omega)}.
$$
We denote by $\cM^s:=\cM^s(\cT_0,L_q(\Omega))$ the class of all
matrices such that
   \be
   \label{Ms}
   |A|_{\cM^s}:= \sup_{n\ge 1} \Big(n^s\sigma_n(A)_{L_q(\Omega)}\Big)
   < \infty.
  \ee

   This accounts for $L_q$ approximability. But
   in our application of the algorithm \COEFF, we need that the matrix $\hat
   A$ is also positive definite to make use of the perturbation
   estimate \eqref{PT1}.  We show later in \S\ref{SS:positive} that if
   we know the bounds \eref{eigenbound} for the eigenvalues of $A$,
   then there is a constant $C_4$ and a piecewise polynomial
   matrix $\hat A$ of degree $\le m_A$ such that
   \be
   \label{prophatA}
   \|A-\hat A\|_{L_q(\Omega)}\le C_4 \sigma_n(A)_{L_q(\Omega)}
   \ee
    where the eigenvalues of $\hat A$ satisfy \eref{eigenbound}
   for some $\hat r$ and $\hat M$ comparable to $r$ and $M$
   respectively;
   \bnew{the constant $C_4$ is independent of $r,M$ and $n$. This
     issue arises of course for $m_A\ge1$ since the best piecewise constant
     approximation of $A$ in $L_q(\Omega)$ preserves both bounds $r$ and $M$.}
            
   In analogy to  \MAIN and \APPROX, 
   we say that  the algorithm \COEFF  has {\it
     class optimal performance} on $\cM^s$ if the number 
   $N(A)$ of elements marked for
   refinement to achieve the tolerance $\veps$ starting from a partition $\cT$ always satisfies
   \be
   \label{coapprox}
   N(A) \le  C_3 |A|^{1/s}_{\cM^s}\veps^{-1/s},
   \ee
  with $C_3$ an absolute constant.  
  Again, we refer to \S\ref{SapproxA} for the construction of such algorithms.  
   
   As with the approximation class $\cA^s$ earlier,
    we get exactly the same approximation classes 
    $\cB^s=\cB^s(\cT_0,H^{-1}(\Omega))$ and
    $\cM^s=\cM^s(\cT_0,L_q(\Omega))$ if we require in addition that the
    partitions are conforming.

\medskip\noindent
{\bf Performance of AFEM.}
Given the approximation classes $\cA^s,\cB^s,\cM^s$,
a goal for performance of an AFEM would be that  
whenever  $u\in\cA^s,f\in\cB^s,A\in\cM^s$, the AFEM produces a
  sequence $\{\cT_k\}_{k\ge 0}$ of nested  triangulations
  ($\cT_{k+1}\ge\cT_k$ for each $k\ge 0$)  such that for $k\ge1$
       \be
       \label{moremodest}
       \|u-U_k\|_{H_0^1 (\Omega)}\le C 
       \big(|u|_{\cA^s} + |f|_{\cB^s} + |A|_{\cM^s} \big)
       (\#\cT_k-\#\cT_0)^{-s},
       \ee
with $C$  a constant depending only on $s$.
The bound \eqref{moremodest} is in the spirit of
Binev et al \cite{04BDD}, Stevenson \cite{Ste7}, and
Casc\'on et al \cite{08CKNS} in that the regularity of the triple 
$(u,f,A)$ enters. It was recently shown in \cite{CDN} that in the case $A=a\ I$,
where $a$ is a piecewise constant function and $I$ the identity
matrix, $u\in\cA^s$ implies $f\in\cB^s$, provided $s<S$. 
\bnew{No such a result exists for $A$, which entails a nonlinear (multiplicative)
  relation with $u$.}

\subsection{$\veps$ -- approximation and class optimal performance}\label{SS:useful}
As already noted  in \S \ref{SS:Performance}, the context on which we
invoke \MAIN is unusual 
in the sense that the diffusion coefficient and the right hand
sides may change between iterations. 
Therefore, to justify \eqref{optmain} in our current setting, we will need some
observations about how it is proved for instance in
\cite{04BDD,Ste06,08CKNS};  see also \cite{NSV:09}.

Let $\hat u \in H^1_0(\Omega)$ be the solution of \eqref{wproblem}
with data $\hat A$ and $\hat f$. 
Let $\hat U=\GAL(\cT,\hat A,\hat f) \in \mathbb V(\cT)$ be the
Galerkin solution subordinate to the input
subdivision $\cT$ and let $\hat \veps := \|\hat u- \hat U\|_{H^1_0(\Omega)}$ 
be the error achieved by the Galerkin solution on $\cT$. 
The control on how many cells are selected by \MAIN is done 
by comparing with the smallest partition $\cT^*$ which achieves accuracy  $\mu
\hat\veps$ for some $0<\mu\le 1$ depending on the D\"orfler marking
parameter and the scaling constants in the upper and lower a
posteriori error estimates \bnew{\cite{Ste7,08CKNS,NSV:09}}. 
That is, if $\cM$ denotes the set of selected (marked) cells for refinement, one
compares $\#\cM$ with $\#(\cT^*\oplus \cT)-\#\cT_0$ to obtain 
\be
\label{only2}
\#\cM\le \tilde C_3|\hat u|_{\cA^s}^{1/s} (\mu \hat\veps)^{-1/s},
\ee
whenever $\hat u \in \cA^s$ and where $\tilde C_3$ is an absolute
constant (depending on $s$).

First, it is important to realize that the argument leading to
\eqref{only2} does not require the
full regularity $\hat u \in \cA^s$ but only that $\hat u$ is an
$\mu\hat\veps$-approximation of order $s$ to some $v \in \cA^s$;
see \S\ref{SS:Performance}.

Second, we note that several sub-iterations within $\MAIN(\cT,\hat A,\hat
f,\veps)$ might be required to achieve the tolerance $\veps$.
However, each sub-iteration selects a number of cells satisfying
\eqref{only2}, \bnew{with $\veps$ instead of $\hat\veps$}, 
and the number of sub-iterations is \bnew{dictated by the ratio $\hat\veps/\veps$; 
see \eqref{inner-iter}.}
Our new \AFEM algorithm will keep this ratio bounded thereby
ensuring that the number of sub-iterations within \MAIN remains uniformly bounded.

 In conclusion, combining Lemma \ref{l:approx} with \eqref{only2},
we realize that $N=N(\hat u)$, 
the number of \bnew{cells} marked for refinement by $\MAIN(\cT,\hat A,\hat
f,\veps)$ to achieve the desired tolerance $\veps$, satisfies
\begin{equation}\label{optmainour}
N \leq C_3 |u|_{\cA^s}^{1/s}\veps^{-1/s}
\end{equation}
provided that $\hat u$ is an $\mu\hat\veps$--approximation of order $s$ to $u\in
\cA^s(\cT_0,H^1_0(\Omega))$ and that each call of \MAIN corresponds to
a ratio $\hat\veps/\veps$ uniformly bounded. This
proportionality constant is absorbed into $C_3$.
We will use this fact in the analysis of our new AFEM algorithm.

\section{AFEM for Discontinuous Diffusion Matrices: \DISC}  
We are  now in the position to formulate our new AFEM which will be denoted by \DISC.
   It will consist of  three main modules  \APPROX,  \COEFF and \MAIN.  While
   the  algorithms \APPROX and \MAIN are standard, we recall that \COEFF
   requires an approximation of $A$ in $L_q(\Omega)$ instead of
   $L_\infty(\Omega)$ for some $q=2p/(p-2)$ where $p$ is such that \CONDP holds.

We assume that the three  algorithms \APPROX, \COEFF, and \MAIN are known to
be class optimal for   all $0<s\le S$  with $S>0$.  The algorithm \DISC inputs an initial
conforming subdivision $\mathcal T_0$, an
initial tolerance $\veps_1$, the matrix $A$ and the right side $f$ for which
we know the solution $u$ of \eqref{wproblem} satisfies $\| \nabla
u\|_{L_p(\Omega)}\leq C_p \| f\|_{W^{-1}(L_p)}$, see \CONDP in
\S\ref{ss:Lp}.   
We now fix constants  $0<\omega,\beta<1$  such that 
\be
\label{constconditions}
\omega\le \frac{r \mu \bnew{C_L}}{\bnew{2} C_2(1+C_p \| f\|_{W^{-1}(L_p(\Omega))})},
\ee
where $\mu\le 1$ is the constant of \S \ref{SS:useful}, \bnew{$C_2$
appears in the uniform bound \eqref{COEFF1} on the eigenvalues and $C_L$ is the lower bound constant in \eqref{e:equiv}.}
  
Given an initial mesh $\cT_0$ and parameters
$\veps_0,\omega,\beta$, the algorithm \DISC sets $k:=0$ and
iterates:
\medskip
\begin{algotab}
\> $[\cT_{k}(f), f_{k}] = \APPROX (\cT_k,f, \omega \eps_k)$\\
\> $[\cT_{k}(A),  A_{k}] = \COEFF (\cT_k(f),A, \omega \eps_k)$\\
\> $[\cT_{k+1},U_{k+1}] = \MAIN (\cT_{k}(A),  A_{k}, f_k, \eps_k/2)$\\
\> $\veps_{k+1} = \beta \veps_{k}$; $k \leftarrow k+1$.
\end{algotab}
\medskip
 
The following theorem shows the optimality of \DISC.  
\begin{theorem}[optimality  of \DISC]\label{t:complexity}
Assume that the three algorithms \APPROX, \COEFF and \MAIN are of class optimal for all $0<s\leq S$ for some $S>0$.
In addition, assume that the right side $f$ is in $ \cB^{s_f}(H^{-1}(\Omega))$ with
$0<s_f\leq S$, that \CONDP holds for some $p>2$ and  that the diffusion matrix $A$ is positive definite, in
$L_\infty(\Omega)$ and in $\cM^{s_A}(L_q(\Omega))$ for $q:=\frac{2p}{p-2}$ and
$0<s_A\leq S$.     
Let $\cT_0$ be the initial subdivision and $U_k \in \mathbb V(\cT_k)$ be the Galerkin solution obtained at the $k$th
iteration of the algorithm \DISC. Then,  
 whenever $u\in\cA^{s_u}(H^1_0(\Omega))$ for
$0<s_u\leq S$, we have for $k\ge1$
\be
\label{mte}
\|u-U_k\|_{H_0^1(\Omega)}\le \bnew{C_U} \veps_{k-1},
\ee
\bnew{where $C_U$ is the upper bound constant in \eqref{e:equiv}}
and  
\be
\label{mt}
\#\cT_{k}-\#\cT_0 \leq C_4 \Big(|u|_{\cA^{s}(H_0^1(\Omega))}^{1/s} +
   |A|_{\cM^{s}(L_q(\Omega))}^{1/s} +
   |f|_{\cB^{s}(H^{-1}(\Omega))}^{1/s}
   \Big) \veps_{k-1}^{-1/s},
\ee
\bnew{with $C_4:=\frac{C_0 C_3\omega^{-1/s}}{1-\beta^{1/s}}$} and $s=\min(s_u,s_A,s_f)$.
\end{theorem}
\begin{proof}
Let us first prove \eref{mte}.   We denote by $u_k$ the solution to
\eref{wproblem} for the diffusion matrix $A_k$ and right side $f_k$.  From
\eref{PT1} and since \CONDP holds, we have
$$
\|u-u_k\|_{H_0^1(\Omega)} 
\le \hat r^{-1}\|f-f_k\|_{H^{-1}(\Omega)}+\hat r^{-1}\|\nabla u\|_{L_p(\Omega)}\|A-A_k\|_{L_q(\Omega)}. 
$$
In addition, the restriction \eref{constconditions} on $\omega$ and the bound
\eqref{COEFF1} on the eigenvalues of $A_k$ lead to
\begin{equation}\label{secondstep2}
\|u-u_k\|_{H_0^1(\Omega)} 
\le  C_2  r^{-1}(1+C_p\| f\|_{W^{-1}(L_p)})\omega\veps_k \le 
 \bnew{\frac{\mu\veps_k C_L}{2}}.
\end{equation}
\bnew{In view of \eqref{e:equiv} \bnew{and \eqref{upper-eps}}, 
$U_{k+1}$ satisfies $\|u_k-U_{k+1}\|_{H_0^1(\Omega)}\le C_U \eta_{\cT_{k+1}}(U_{k+1},\bnew{A_k,f_k,}\cT_{k+1}) \le C_U \veps_k/2$.
Moreover, we have $C_L\leq C_U$ and
$\mu\le 1$ so that the triangle inequality yields \eqref{mte}
$$
\|u-U_{k+1}\|_{H_0^1(\Omega)} \leq C_U \veps_k,\qquad\forall k\ge0.
$$
}

Next, we prove \eref{mt}.
At each step $j$ of the algorithm, the new partition $\cT_j$ is generated from
$\cT_{j-1}$ by selecting cells for refinement and possibly others to ensure the
conformity of $\cT_{j}$.
We denote by $N_j(f)$, $N_j(A)$, and \bnew{$N_j(u)$} the number of cells selected for refinement
by the routines \APPROX, \COEFF, and \MAIN respectively.
The bound \eqref{nv1} accounts for the extra refinements to
create conforming subdivisions, namely
$$
\#\cT_k - \#\cT_0 \leq C_0 \sum_{j=0}^{k-1}\big(N_j(f)+N_j(A)+\bnew{N_j(u)}\big).
$$

The class optimality assumptions of \APPROX and \COEFF directly imply that
$$
N_j(f) \leq C_3  
\bnew{|f|^{1/{s}}_{\cB^{s}(H^{-1}(\Omega))}
\big(\omega\veps_j\big)^{-1/s}
},
\quad
\text{and} \quad
N_j(A) \leq C_3  
\bnew{|A|^{1/{s}}_{\cM^{s}(L_q(\Omega))}
\big(\omega\veps_j\big)^{-1/s},
}
$$
\bnew{because $s\le s_f,s_A$.}
We cannot directly use that $u\in\cA^s$ to bound \bnew{$N_j:=N_j(u)$}, because
  $N_j$ is dictated by the inherent scales of $u_j$, the solution of
\eqref{wproblem} with diffusion coefficient $A_j$ and right hand side
$f_j$. Let $U_j(A)=\GAL(\cT_j(A),A_j,f_j)$, set 
\bnew{
$\hat\veps_j:= \eta_{\cT_j(A)}(U_j(A),A_j,f_j;\cT_j(A))$ so that from \eqref{e:equiv} we have $\| u_j - U_j(A)\|_{H^1_0(\Omega)} \geq C_L \hat \veps_j$,} and
assume that $\hat\veps_j>\veps_j/2$ for otherwise the call 
$\MAIN(\cT_j(A), A_j, f_j,  \veps_j/2)$ is skipped and $N_j=0$.

In view of \bnew{\eqref{inner-iter} and the discussion of
  \S\ref{SS:useful}}, estimate \eqref{optmainour}
is valid upon proving that \bnew{the ratio $\hat\veps_j/\veps_j$} is uniformly
bounded with respect to the iteration \bnew{counter} $j$ and that $u_j$ is an
$\mu \bnew{C_L\veps_j}$-approximation of order $s$ to $u$.
The latter is direct consequence of  the estimate $\|u-u_j\|_{H^1_0(\Omega)}\leq
\mu \veps_j \bnew{C_L}/2$ given in \eqref{secondstep2} and Lemma
\ref{l:approx}, so that only the uniform bound on 
\bnew{$\hat\veps_j/\veps_j$} remains to be proved.

We recall the Galerkin projection property
$$
\| A_j^{1/2} \nabla (u_j- U_j(A))\|_{L_2(\Omega)} \leq 
\| A_j^{1/2} \nabla (u_j-V)\|_{L_2(\Omega)} 
$$
holds for any  $V \in \mathbb V(\cT_j(A))$ and in particular for
$U_j\in \mathbb V(\cT_j) \subset \mathbb V(\cT_{j}(A))$ because
$\cT_j(A)$ is a refinement of $\cT_j$.
If $r_j=\gamma_{\min}(A_j)$ and $M_j=\gamma_{\max}(A_j)$ denote the minimal
and maximal eigenvalues of $A_j$, then the above Galerkin projection
property, \bnew{the lower bound in \eqref{e:equiv}, and \eqref{COEFF1}} yield
\begin{equation*}
\begin{aligned}
\hat\veps_j & \leq \bnew{C_L^{-1}} \|u_j-U_j(A)\|_{H^1_0(\Omega)} \leq 
\bnew{C_L^{-1}} (M_j/r_j)^{1/2} \|u_j-U_j\|_{H^1_0(\Omega)} \\
&\leq \bnew{C_L^{-1} C_2(M/r)^{1/2}} \left(\|u-U_j\|_{H^1_0(\Omega)}+\| u-u_j \|_{H^1_0(\Omega)}  \right).
\end{aligned}
\end{equation*}
 Combining \eqref{mte} and $\veps_j=\beta
\veps_{j-1}$, 
together with \eqref{secondstep2}, implies the
desired bound
$$
\hat\veps_j/\veps_j 
\le \bnew{C_L^{-1}C_2(M/r)^{1/2}\big(C_U/\beta+\mu C_L/2\big).}
$$

The argument given in \S\ref{SS:useful} guarantees the bound
\eqref{optmainour}, namely 
$$
N_j \leq C_3 |u|_{\cA^{s_u}(H^1_0(\Omega))}^{1/2} \veps_j^{-1/s}.
$$
Gathering the
bounds on  $N_j(f)$, $N_j(A)$ and \bnew{$N_j(u):=N_j$}, 
and using that $\omega<1$, we deduce
$$
\#\cT_k - \#\cT_0 \leq C_0 C_3 \bnew{\omega^{-1/s}} \Big(|A|_{\cM^{s}(L_q(\Omega))}^{1/s} +
   |f|_{\cB^{s}(H^{-1}(\Omega))}^{1/s}
   +|u|_{\cA^{s}(H_0^1(\Omega))}^{1/s} 
  \Big) \sum_{j=0}^{k-1}\veps_j^{-1/s}.
$$
The desired estimate \eqref{mt} is obtained after writing
$\veps_j=\beta^{k-j}\veps_k$ and recalling that $0<\beta<1$ so that
 $\sum_{j=0}^{k-1}\beta^{j/s} \leq (1-\beta^{1/s})^{-1}$.
  \end{proof}

Note that, we could as well state the conclusion of Theorem \ref{t:complexity} as
$$
\|u-U_k\|_{H_0^1(\Omega))}\le C\Big(
|u|_{\cA^{s}(H_0^1(\Omega))}+|f|_{\cB^{s}(H^{-1}(\Omega))}+|A|_{\cM^{s}(L_q(\Omega))}
\Big) (\#\cT_k-\#\cT_0)^{-s},
$$
 for a constant $C$ independent of $k$, whence \DISC has optimal
performance according to \eqref{moremodest}.

\begin{remark}[the case $s<s_u$]\label{ss:optimality_disc}
\rm
We briefly discuss why the decay rate  $s=\min(s_u,s_f,s_A)$ cannot be
improved in \eqref{mt} to $s_u$ (the optimal rate for the approximation of $u \in
\cA^{s_u}$) by any algorithm using approximations $\hat A$ of $A$ and $\hat f$ of $f$.
We focus on the effect of the diffusion coefficient $A$, assuming the right hand
side $f$ is exactly captured by the initial triangulation  $\cT_0$, since 
a somewhat simpler argument holds for the approximation of $f$.

 The approximation of $u$ by $\hat U$, the Galerkin solution with diffusion
coefficient $\hat A$, cannot be better than that of $A$ by $\hat A$.
Indeed, there are two constants $c$ and $C$ such that for any $\delta >0$
\begin{equation}\label{e:noise}
c \delta \leq \sup_{A_1,A_2 \in B(A,\delta)} || u_{A_1} - u_{A_2}||_{H^1_0(\Omega)}
\leq C  \delta,
\end{equation}
where
$u_{A_i} \in H^1_0(\Omega)$, $i=1,2$, are the weak solutions of
$
-\text{div}(A_i \nabla u_{A_i}) = f
$
and for $\delta>0$
$$
B(A,\delta):=\left\lbrace d\times d \ \text{positive matrices} \ B \ | \
  ||A-B||_{L_\pA(\Omega)} \leq  \delta \right\rbrace.
$$
While the right inequality is a direct consequence of the perturbation theorem
(Theorem \ref{T:perturbation}), the left inequality is obtained by the
particular choice $A_1=A$ and $A_2=\big(1+\frac{\delta}{||A||_{L_\pA(\Omega)}}\big)^{-1}A$.
In fact, this choice implies that $A_1,A_2\in B(A,\delta)$ and
$
u_{A_2} = \big(1+\frac{\delta}{||A||_{L_\pA(\Omega)}}\big)u_{A_1}.
$
Therefore
\begin{equation*}
\begin{aligned}
||u_{A_1}-u_{A_2} ||_{H^1_0(\Omega)} &=
\frac{\delta}{||A||_{L_\pA(\Omega)}} || u_{A_1} ||_{H^1_0(\Omega)} 
\\
& \geq
\frac{\delta}{M^{1/2} \|A\|_{L_\pA(\Omega)}} || A^{1/2}
\nabla u_{A_1} ||_{L_2(\Omega)}
\geq \frac{r^{1/2} || f ||_{H^{-1}(\Omega)}}{M^{1/2} ||A||_{L_\pA(\Omega)}}
~\delta, 
\end{aligned}
\end{equation*}
where $0<r<M<\infty$ are the lower and upper bounds for the
\bnew{eigenvalues} of $A$.
\end{remark}

\section {Algorithms \APPROX and \COEFF}\label{SapproxA}
We have proven \bnew{the optimality of \DISC in Theorem \ref{t:complexity}}
provided the subroutines \APPROX and \COEFF are themselves optimal.    In this section, we discuss what is known about the construction of optimal algorithms for \APPROX and \COEFF.  
\bnew{Recall that \APPROX constructs an approximation of $f$ in $H^{-1}$ while $\COEFF$ an approximation of $A$ in $L_q$.}
 A construction of algorithms of this type can be made at two levels.
The first, which we shall call the {\it theoretical level},  addresses this problem by assuming we have complete knowledge of $f$ or $A$ and anything we need about
them can be computed free of cost.  This would be the case for example if $f$
and $A$ were known piecewise  smooth functions on some fixed known partition (which
could be unrelated to the initial partition). The second level, which we call the {\it practical level}, realizes that in most  applications of AFEMs, we do not precisely know $f$ or $A$ but what we can do,  for example, is compute for any chosen query point $x$ the value of these functions to high precision.  It is obviously easier to construct theoretical algorithms and we shall primarily discuss this issue.

\subsection{Optimal algorithms for adaptive approximation of  a function}
\label{SS:approxfunctions}
  The study and construction of algorithms like  \APPROX is a central subject not only in adaptive finite element methods but also in approximation theory.   
These algorithms are needed in all AFEMs.   The present paper is not intended to advance  this particular subject.  Rather, we want only to give
an overview of what is known about such algorithms both at the theoretical and practical level.  
We begin by discussing
$L_q$ approximation.

Two adaptive algorithms were introduced in \cite{BD} for approximating functions and were proven to be optimal in several settings.
These algorithms   are built on  local error estimators.
\bnew{Given $g\in L_q(\Omega)$, we define the local $L_q$ error in a
polyhedral cell $T$ by}
\be
\label{local error}
E(T):=E(g,T)_{L_q(T)}:=\inf_{P\in\mathbb P_{m}(T)}\|g-P\|_{L_q(T)},
\ee
where $\mathbb P_{m}(T)$ is the space of polynomials of degree $\leq m$ over $T$.  Given a
partition $\cT$, the best approximation to $g$ by piecewise
polynomials of degree $\leq m$ is obtained by taking the best polynomial approximation  $P_T$  to $g$ on $T$ for each $T\in\cT$ and then
\be
\label{ba}
S_\cT:=\sum_{T\in\cT}P_T\chi_T,
\ee
where $\chi_T$ is the
characteristic function of $T$.
Its global error is
\be
\label{ba1}
\cE(g,\cT) := \|g-S_\cT\|_{L_q(\Omega)}
=\Big(\sum_{T\in\cT} E(g,T)_{L_q(T)}^q\Big)^{1/q}.
\ee
So finding good approximations to $g$  reduces to 
finding good partitions $\cT$ with small cardinality.

The algorithms in \cite{BD} adaptively build   partitions   by examining the local errors $E(g,T)_{L_q(\Omega)}$ for $T$ in the current partition and then refining some of these cells based not only on the size of this error but also the past history.  The procedure penalizes cells which arise from previous refinements that did not
significantly reduce the error.  The main result of \cite{BD} is that for $0<q<\infty$,  these two algorithms start from $\cT_0$ and construct partitions $\cT_n$, $n=1,2,\dots$,
such that $\cT_n\in\fT_{c n}$ and
\be
\label{BDoptimal}
\cE(g,\cT_n)\le C \ \sigma_{ n}(g)_{L_q(\Omega)},\quad n=1,2,\dots,
\ee
where $c>1$ and $C>1$ are fixed constants. In view of \eqref{e:overlay}, it  follows   that these algorithms are both optimal for $L_q$ approximation, $1\le q<\infty$, for all $s>0$.

While the above algorithms are optimal for $L_q$ approximation, they are often replaced by the simpler strategy of marking and refining only the cells with largest local error.   We describe and discuss one of these strategies known as the   {\it greedy algorithm}. Given {\it any}  
refinement $\cT$ of the initial mesh $\cT_0$, the procedure
$\cT' =\cT'(\veps)= \GREEDY (\cT,g,\veps)$ constructs a \bnew{conforming}
refinement $\cT'$ of $\cT$ such that
$\cE(g,\cT') \le \veps$.

To describe the algorithm, we first recall that  the
bisection rules of \S \ref{SS:partitions} define a unique forest $\fT$
emanating from $\cT_0$.  The elements in this forest can be given a unique lexicographic ordering. 
The algorithm reads:
\medskip
\begin{algotab}
\> $\cT(\veps) = \GREEDY(\cT,g,\veps)$ \\
\> $\cT' = \cT$; \\
\> {\tt while} $\cE(g,\cT') > \veps$ \\
\> \>  $T := \argmax \big\{E(g,T) \, : \, T\in\cT'\big\}$; \\
\> \>  $\cT' := {\REFINE} (\cT', T)$; \\
\> {\tt end while}\\
\> $\cT(\veps)=\CONF(\cT')$      
\end{algotab}
\medskip
\bnew{The procedure $\cT' = \REFINE (\cT,T)$  replaces $T$ by its two
children. The selection of $T$ is done by choosing the smallest
lexicographic $T$ to break ties. Therefore,}
we see that $\GREEDY$ chooses an element $T\in\cT$ with largest
error $E(g,T)$  and replaces $T$ by its two children
to produce the next \bnew{non-conforming} refinement $\cT'$ until
the error $\cE(g,\cT')$ is below the prescribed tolerance $\veps$. 
Upon exiting the while loop, additional refinements are made on $\cT'$ by \CONF to obtain the smallest  conforming partition $\cT(\veps)$ which contains $\cT'$.

An important property of the local error 
$E(T)$ is its monotonicity
\begin{equation}\label{local-monotonicity}
E(T_1)^q + E(T_2)^q \le E(T)^q \qquad\forall T\in\cT, 
\end{equation}
where $T_1,T_2$ are the two children of $T$. This leads to a global
monotonicity property
\begin{equation}\label{e:monotone_greedy}
\cE(g,\cT') \le \cE(g,\cT)
\end{equation}
for all refinements $\cT'$ of $\cT$ whether conforming or not.

While the greedy algorithm is not proven to be optimal in the sense of giving the rate $O(n^{-s})$ for the entire class
$\cB^s(L_q(\Omega))$,   it is known to be optimal on subclasses of $\cB^s$.  For example, it is known that any finite ball
in the Besov space $B^s_\infty(L_\tau(\Omega))$ with $s/d > 1/\tau -
1/q$ and $0<s\le m+1$ is contained in \bnew{$\cB^{s/d}(L_q(\Omega))$}.
The following proposition shows that the greedy algorithm is optimal
on these Besov balls.

\begin{prop}[performance of $\GREEDY$]\label{P:class-opt}
If $g\in B^s_\infty(L_\tau(\Omega))$ with $s/d > 1/\tau - 1/q$ and $0<s\le m+1$, then $\GREEDY$
terminates in a finite number of steps
and marks a total number of elements $N(g):=\#\cT'-\#\cT$ satisfying
\begin{equation}\label{n-greedy}
N(g) \le C_3 |g|_{B_\infty^s(L_\tau(\Omega))}^{d/s} \, \veps^{-d/s}
\end{equation}
with a constant $C_3$ depending only on $\cT_0, |\Omega|, \tau, s$ and
$q$. Therefore, $g\in\cB^{s/d}(\cT_0,L_q(\Omega))$ with
$|g|_{\cB^{s/d}(\cT_0,L_q(\Omega))} \lesssim |g|_{B_\infty^s(L_\tau(\Omega))}$.
\end{prop}

Results of this type have a long history beginning with the famous theorems of Birman and Solomyak \cite{BirSol67} for Sobolev spaces, \cite{BDDP} for Besov spaces, and \cite{CDDD} for the analogous wavelet tree approximation; see also the expositions in \cite{BCMMN,CDN,NSV:09}. If $s/d=1/\tau-1/q$, it turns out that there are
functions in $B_\infty^s(L_\tau(\Omega))$ which are not in 
$\cB^{s/d}(\cT_0,L_q(\Omega))$ \cite{BDDP}. Therefore, the assumption 
$s/d>1/\tau-1/q$ of Proposition \ref{P:class-opt} 
is sharp in the Besov scale of spaces to obtain
$B^s_\infty(L_\tau(\Omega))\subset\cB^{s/d}(\cT_0,L_q(\Omega))$.
We may thus say that $\GREEDY$
has {\it near class optimal performance} in
$\cB^{s/d}(\cT_0,L_q(\Omega))$.

Proposition \ref{P:class-opt} differs from these previous results in
the marking of only one cell at each iteration.  However, its
proof follows the same reasoning as that given in \cite{BDDP} (see
also \cite{CDN,NSV:09}) except for the following important point.   
  The proofs in the literature assume that
the greedy algorithm begins with the initial partition $\cT_0$ and not a general partition $\cT$
as stated in the proposition.  This is an important distinction  since our algorithm \DISC is applied to general $\cT$.
 We now give a simple argument that shows that 
   the number of elements $N(g)=N(\cT,g)$
marked by $\GREEDY(\cT,g,\veps)$ starting from $\cT$ satisfies 
\be
\label{satisfies}
N(g)\le N=N(\cT_0,g)
\ee
and thus \eqref{n-greedy}. 
We first recall that the
bisection rules of \S \ref{SS:partitions} define a unique forest $\fT$
emanating from $\cT_0$ and a unique sequence of elements 
$\{T_i\}_{i=1}^{N}\subset\fT$ created by $\GREEDY(\cT_0,\veps)$.
Let $\cT'_{i+1} = \REFINE(\cT'_i,T_i)$ be the intermediate
subdivisions obtained within $\GREEDY(\cT_0,\veps)$ to refine
$T_i, 1\le i\le N$.
  Let $\Lambda$ be the set of indices $j\in\{1,\dots,N\}$ such that
  $T_j$ is never refined in the process to create $\cT$, i.e $T_j$ is either an element of $\cT$ or a successor of an element of $\cT$.
If $\Lambda=\emptyset$, then $\cT$ is a refinement of
$\cT'_{N+1}$, whence $N(g)=0$ and we have nothing to prove. 
 If $\Lambda \neq \emptyset$,
we let $j$ be the smallest index in $\Lambda$ and note that
$T_j\in\cT'_{j-1}$ with $\cT'_0=\cT_0$.
The definition of $\Lambda$ in conjunction with the minimality of
$\cT'_{j-1}$ implies that $\cT$ is a refinement of $\cT'_{j-1}$.
Since $T_j$ cannot be a successor of an element of $\cT$, because of
the definition of $T_j$ and the monotonicity property
\eqref{local-monotonicity}, we thus infer that
$T_j$ is an element of $\cT$.
This ensures that $T_j$ is the element with largest local error 
(with lexicographic criteria to break ties) among the elements of
$\cT$, and is thus the element selected by $\GREEDY(\cT,\veps)$.
Therefore, $\GREEDY(\cT,\veps)$ chooses in order the elements
$T_i$, $i\in\Lambda$, and stops when it exhausts $\Lambda$ if not
before, thereby leading to \eref{satisfies}.

Finally, let us note that a similar analysis can be given for the
construction of optimal algorithms for approximating the right hand
side $f$ in the $H^{-1}(\Omega)$ norm,
except that $H^{-1}(\Omega)$ is not a local norm. Since this is 
reported on in detail in \cite{CDN}
we do not discuss this further here. 

  \subsection{Optimal algorithms for \COEFF}\label{ss:coeff}
 Given an integer $m$, we recall the space $\Sigma_n:=\Sigma_n^m$ of
matrix valued piecewise polynomial functions of degree $\le m$, 
the error $\sigma_n(A)_{L_q(\Omega)}$, and the approximation classes $\cM^s(L_q(\Omega))$  that were   introduced in \S  \ref{SS:Performance}.  
Assume that $A\in L_\infty(\Omega)$ 
is a  positive definite matrix valued function whose eigenvalues satisfy
\be
\label{ev1}
r\le \lambda_{\min}(A)\le \lambda_{\max}(A)\le M.
\ee
 This is equivalent to
 \be
 \label{ev2}
r\le  \sum_{i,j=1}^da_{ij}(x)z_iz_j= z^tA(x)z\le M,\quad |z|=1.
\ee
The construction of an algorithm \COEFF to approximate $A$ by elements
$B\in\Sigma_n^m$ has two components.  The first one is to find good
approximants $B\in L_q(\Omega)$ for $q<\infty$.  The second issue is
to ensure that $B$ is also positive definite. We study the latter in
\S \ref{SS:positive} and the former in \S \ref{SS:approxA}.  

\subsubsection{ Enforcing positive definiteness}\label{SS:positive}
 High order approximations $B \in \Sigma_n^m$ of $A$, namely
$m>0$, may not be positive definite.
We now show how to adjust $B$ to make it 
positive definite without degrading its approximation to $A$.

\begin{prop}[enforcing positive definiteness locally]
\label{constlemma1}
Let $A$ be a symmetric positive definite matrix valued function in $L_{q}(\Omega)$ whose eigenvalues are in $[r,M]$ with $M\ge 1$. Let $T\subset \Omega$ be  any polyhedral cell  which may arise from newest vertex bisection applied to $\cT_0$.  If   there is a matrix valued function $B$ which is a polynomial of degree $\le m$ on $T$ and satisfies%
$$
\|A-B\|_{L_q(T)}\le \veps,
$$
 then there is a positive definite matrix $\wt B$ whose entries are also
 polynomials of degree $\leq m$ such that the
eigenvalues of $\wt B$ satisfy 
\be
r/2\le \lambda_{\min}(\wt B)\le \Lambda_{\max}(\wt B) \le CM
\ee
and 
$$
 \|A-\wt B\|_{L_q(T)}\le  \wt C\veps,
$$
where $ C$ depends only on $d,m$ and the initial partition $\cT_0$ 
and $\wt C$ depends additionally on $M/r$.
\end{prop}
\begin{proof}  
Let us begin with  an inverse (or Bernstein) inequality:
there is a constant $C_1$, depending on $m, d$ and $\cT_0$, such
that for any polynomial $P$ of degree $\leq m$ in $d$ variables and any polyhedron $T$
which arises from newest vertex bisection, we have \cite{BS08}
\be
\label{Bern1}
\|\nabla P\|_{L_\infty(T)}\le C_1\|P\|_{L_\infty(T)}|T|^{-1/d}.
\ee

Now, given $A$ and the approximation $B$, let
$$
M_0:=\sup_{x\in T}\sup_{|y|=1}y^tB(x)y.
$$
\bnew{We first  show that the statement is true for $M_0> CM$, for a suitable
constant $C$.}  We leave the value of $C\ge 1$ open at this stage and
derive restrictions on $C$ as we proceed.    
If $M_0 > CM$, then there is a $y$ with $|y|=1$ and an $x_0\in T$ such
that $y^tB(x_0)y=M_0 > CM$.  We fix this $y$ and consider the function
$a(x):=y^tA(x)y$ and the polynomial $P(x):=y^tB(x)y$ of degree $m$.  Notice that
$\|P\|_{L_\infty(T)}=M_0$ and  
\be
\label{notice1}
\|a-P\|_{L_q(T)} \le \|A-B\|_{L_q(T)}\le \veps.
\ee
In view of \eqref{Bern1}, we have
$$
|P(x)-P(x_0)|\le C_1M_0|T|^{-1/d}|x-x_0|\le   M_0/2,\quad |x-x_0|\le |T|^{1/d}/(2C_1).
$$
Let $T_0$ be the set of $x\in T$ such that  $|x-x_0|\le
|T|^{1/d}/(2C_1)$. Then $P(x)\ge M_0/2$ on $T_0$  and hence
$|a(x)-P(x)|\ge M_0/4$ on $T_0$ provided $C\ge4$ because 
$0<a(x)\le M < M_0/C$. 
Since $T_0$  has measure $\ge c|T| $, with $c<1$ depending only on $d$ and $\cT_0$, we obtain 
\beqn
\label{pol3}
M_0 (c|T|)^{1/q}  \le 4\|a-P\|_{L_{q}(T)}\le 4\veps.
  \eeqn
 If we define $\wt B:=r I$ on $T$, with $I$ the identity,
then using \eref{pol3} we obtain
 \beqn
 \label{pol4}
\|A-\wt B\|_{L_{q}(T)}\le \big(M+r\big)|T|^{1/q}
\le 2M|T|^{1/q}\le \big(2 M_0/C \big)|T|^{1/q} \le  (M_0/4) (c|T|)^{1/q}
 \le \veps,
 \eeqn
  provided $C$ is chosen large enough so that $C^{-1}c^{-1/q}\le 1/8$.   
This implies $C\ge 4$ and fixes the value of $C$. Thus, we have satisfied the lemma in the case $M_0>CM$.
  
 \bnew{We now discuss the case $M_0\le CM$ and consider}
 \be
 \label{min}
 \mu:=\inf_{x\in T}\inf_{|y|=1}y^tB(x)y.
 \ee
 If $\mu \ge r /2$, we have nothing to prove. So, we assume that $\mu<r/2$ and fix $y_0$ with $|y_0|=1$ and $x_0\in T$  such that $P(x):=y_0^tB(x)y_0$ assumes the value $\mu$ at $x_0$.    We then have from \eqref{Bern1}
 \beqn
 \label{pol41}
 |P(x)-P(x_0)|\le C_1 M_0|T|^{-1/d}|x-x_0|\le C_1CM|T|^{-1/d}|x-x_0|,\quad x\in T.
 \eeqn
 Let $T_0$ be the set of $x\in T$ such that $|x-x_0|\leq r |T|^{1/d}/(4C_1CM)$.
 Notice that $|T_0| \geq cr^dC_1^{-d}C^{-d}M^{-d}|T|$ for some constant $c$ only
 depending on $d$ and $\cT_0$, and $|P(x)-P(x_0)|\le r/4$.
 It follows that $P(x)\le \mu+r/4$ on the subset  $T_0$.  Thus, for 
$a(x):=y_0^tA(x)y_0\ge r$, this   gives that $a(x)-P(x)\ge \frac{3r}{4}-\mu$, $x\in T_0$ and therefore 
 \beqn
 \label{pol421}
|T_0|^{1/q}\Big(\frac{3r}{4}-\mu\Big)\le \|a-P\|_{L_{q}(T)}\le \veps.
\eeqn
 We now define $\wt B=B+(\frac{3}{4}r-\mu) I$, so that 
 $$\bar y^t\bar B(x)y\ge \frac{3}{4}r,\quad |y|=1,\ x\in T,$$
  and 
 \begin{equation*}
 \|A-\wt B\|_{L_{q}(T)}\le
 \|A-B\|_{L_{q}(T)}+\Big(\frac{3}{4}r-\mu\Big)|T|^{1/q}\le \veps
 +\Big (\frac{3}{4}r-\mu\Big)|T_0|^{1/q}\big( |T|
 |T_0|^{-1}\big)^{1/q} \le \wt C\veps,
 \end{equation*}
where we have used  \eref{pol421} and the value of the measure of $T_0$.  
\end{proof}

\begin{remark}[form of $\wt B$]
\label{rem1}
\rm
In the setting of Proposition \ref {constlemma1}, the matrix $\wt B$ takes 
one of three forms: (i) $\wt B=B$, (ii) $\wt B= r I$, 
(iii) $\wt B=B+\alpha  I$, for some $\alpha>0$.   
To convert this recipe into a numerical procedure we need first
the approximation $B$ of $A$. We construct $B$ in \S \ref{SS:approxA}.
\end{remark}

We can now prove the following theorem which shows that there is no 
essential loss of global accuracy by requiring that the 
approximation $B$ to $A$ be uniformly positive definite.

\begin{theorem}[enforcing positive definiteness globally]
\label{T:approxA}
Let $A$ be a positive definite matrix valued function in $L_q(\Omega)$  whose eigenvalues are all in $[r,M]$,  for each $x\in \Omega$, with $M\ge 1$.  If $B\in\Sigma_n^m$ satisfies  
$$
\|A-B\|_{L_q(\Omega)}\le \veps,
$$
 then there is an $\wt B\in \Sigma_n^m$ whose eigenvalues are in
 $[r/2,CM]$ for all $x\in\Omega$ and satisfies
 \beqn
\label{pol2}
\|A-\wt B\|_{L_q(\Omega)}\le  \wt C\veps,
\eeqn
where   $C,\tilde C$ are  the constants of Proposition \ref{constlemma1}. 
\end{theorem}
\begin{proof}   Let $B=\sum_{T\in\cT} B_T\chi_T$ where each of the matrices
$B_T$ have polynomial entries of degree $\leq m$ and $\chi_T$ is the
characteristic function of $T$.  If  $\wt B_T$ is the matrix
from Proposition \eqref{constlemma1} applied to $A$ and $B_T$ on $T$, then   
$\wt B(x):=\sum_{T\in\cT}\wt B_T(x)\chi_T(x)$ has its eigenvalues in $[r/2,CM]$ for each $x\in\Omega$.
The error estimate \eref{pol2} follows from $\|A-\wt B_T\|_{L_
  q(T)}\le \wt C \|A- B_T\|_{L_ q(T)}$ for each $T\in\cT$.
\end{proof}

 If we define $\wt \sigma_n(A)_{L_q(\Omega)}$ in the same way as
$\sigma _n$, except that we require that the approximating matrices 
are positive definite, then Theorem \ref{T:approxA} implies 
$\tilde \sigma_n(A)_{L_q(\Omega)}\le \tilde
C\sigma_n(A)_{L_q(\Omega)}$ for all $n\ge 1$.

\subsubsection {Algorithms for approximating $A$}\label{SS:approxA}
 In view of Theorem \ref{T:approxA}, we now concentrate on 
approximating $A$ in $L_q(\Omega)$ without preserving positive definiteness.
Let us first observe that approximating $A$ by elements from $\Sigma_n^m$ is simply a matter of approximating its entries.  For any $d\times d$ matrix $B=(b_{i,j})$, we have
that its spectral norm does not exceed $ \sum_{i,j}|b_{i,j}|$ and for any $i,j$ it is at least as large as $|b_{i,j}|$.   Hence,
\be
\label{Dmatrix}
\max_{i,j}\|b_{i,j}\|_{L_q(\Omega)}\le \|B\|_{L_q(\Omega)}\le \sum_{i,j}\|b_{i,j}\|_{L_q(\Omega)}\le d^2 \max_{i,j}\|b_{i,j}\|_{L_q(\Omega)}.
\ee
It follows that approximating $A$ in $L_q(\Omega)$  by elements of $\Sigma_n^m$
is equivalent to approximating its entries  $a_{i,j}$ in $L_q(\Omega)$ by
piecewise polynomials of degree $\leq m$.   Moreover, note that 
$A\in \cM^s(\cT_0,L_q(\Omega))$ is equivalent to each of the entries 
$a_{i,j}$ being in $\cB^s(\cT_0,L_q(\Omega))$.

The analysis given above means that the construction of optimal algorithms for \COEFF follow from the construction
of optimal algorithms for functions in $L_q$ as discussed in \S \ref{SS:approxfunctions}.
If we are able to compute the local error $E(a_{ij},T)_{L_q}$  for each coefficient $a_{ij}$ and each  \bnew{element} $T$, then
we can construct a (near) class optimal algorithm $\COEFF$
for $\cM^s(\cT_0,L_q(\Omega))$.
Moreover, $\COEFF$ guarantees that the approximations are positive
definite and satisfy \eref{COEFF1}; see Proposition \ref{constlemma1}
and Remark \ref{rem1}.

\section{Numerical Experiments}\label{s:numerics}

We present two numerical experiments, computed with bilinear
elements within {\tt deal.II} \cite{BHK:07}, that explore the applicability
and limitations of our theory.
We use the quad-refinement strategy, as studied in 
\cite[Section 6]{BN:10}, instead of newest vertex bisections.
The initial partition $\cT_0$ of $\Omega$ is thus made of quadrilaterals
for dimension $d=2$ and refinements
of $\cT_0$ are performed using the quad-refinement strategy imposing at most one
hanging node per edge,  as implemented in {\tt deal.II} 
\cite{BHK:07}.  
We recall that our theory is valid as well for quadrilateral or hexahedral
subdivisions with limited
amount of hanging nodes \bnew{per edge}. We refer to \cite[Section 6]{BN:10} for details about
refinement rules and computational complexity. 

Before starting with the experiments we comment on the choice of the
Lebesgue exponent $q=\frac{2p}{p-2}$ needed for the approximation of $A$; see 
Section \ref{ss:coeff}. Since we do not know $p$ in general, the
question arises how to determine $q$ in practice. We exploit the fact
that both $A$ and its piecewise polynomial approximation $\hat A$ are
uniformly bounded in $L_\infty(\Omega)$, say by a constant $C_A$, 
to simply select $\hat A$ as
the best approximation in $L_2(\Omega)$, computed elementwise,
and employ the interpolation estimate
\begin{equation}\label{L2-best}
\| A-\hat A\|_{L_q(\Omega)} \leq (2C_A)^{2/p} \| A-\hat A\|_{L_2(\Omega)}^{2/q}, 
\qquad 2\leq q \leq \infty.
\end{equation}
The perturbation estimate \eqref{PT1} thus reduces to
$$
\|u-\hat u\|_{H_0^1(\Omega)}\le  \hat r^{-1}\|f-\hat
f\|_{H^{-1}(\Omega)}+ \frac{(2C_A)^{2/p}}{\hat r} \|\nabla u\|_{L_{p}(\Omega)} \|A-\hat
A\|_{L_2(\Omega)}^{2/q},
$$
thereby justifying a universal choice of $\hat A$ 
regardless of the values of $p$ and $q$. 
Notice, however, that we would need the sufficient condition
$A\in\mathcal M^{\frac{qs_A}{2}}(L_2(\Omega))$
for $A\in \mathcal M^{s_A}(L_q(\Omega))$ and $2\leq q <\infty$;
therefore this choice may not always preserve the decay rate of
$\|A-\hat A\|_{L_q(\Omega)}$. A practically important exception occurs
when $A$ is piecewise constant over a finite number of pieces 
with jumps across a Lipschitz curve, since then
$$
\|A-\hat A\|_{L_q(\Omega)} \approx 
\big|\lbrace x \in \Omega  : \ A(x) \ne \hat A(x) \rbrace \big|^{1/q}
\qquad\Rightarrow\qquad
\|A-\hat A\|_{L_q(\Omega)} \approx \|A -\hat A\|_{L_2(\Omega)}^{2/q}.
$$
This in turn guarantees no loss in the convergence rate for $A$.
In the subsequent
numerical experiments, $A$ is piecewise constant and we thus utilize
piecewise constant aproximation for both the diffusion matrix
$A$ and the right hand side $f$ by their meanvalues $\bar A$
and $\bar f$, which is consistent with bilinears for $u$. 
We point out that both $A$ and $\bar A$ share 
the same spectral bounds.

\subsection{Test 1: L-shaped Domain}
We first examine $\DISC$ with the best possible choice $q=2$.
We consider the L-shaped domain $\Omega = [-5,5]\times [-5,5] \setminus
[0,5]\times [0,5]$.
We use $(\rho,\delta)$ to denote the polar
coordinate from the origin $(0,0)$.
The diffusion tensor is taken to be $A=a \ I$, where $I$ is the
$2\times 2$ identity matrix,
 
$$
a(\rho)=\left\lbrace \begin{array}{ll}
1 & \text{if} \qquad \rho \leq  \rho_0,\\
\mu & \text{otherwise},
\end{array}\right.
$$
and $\rho_0=2\sqrt{2}$, $\mu=5$.
Define $v(\rho,\delta) := \rho^{2/3}\sin(2\delta /3)$ to be the standard solution on the
L-shaped.
The exact solution engineered to illustrate the performance of \DISC is the standard singular solution for the L-shaped
domain when $\rho \leq \rho_0$ and a linear extension in the radial direction when
$\rho>\rho_0$, \bnew{namely,}
$$
u(\rho,\delta) = \left\lbrace \begin{array}{ll}
v(\rho,\delta) & \text{if} \qquad \rho \leq  \rho_0,\\
v(\rho_0,\delta)+\frac 2 {3\mu} \rho_0^{1/3} \sin(2\delta/3)(\rho-\rho_0)  
& \text{otherwise}.
\end{array}\right.
$$
Notice that $f:=-\text{div}(A\nabla u) \in
L^2(\Omega)$ and the spectral bounds of $A$ and $\bar A$ are $r=1,M=5$,
by construction.

We emphasize that the discontinuity of $A$ is never
matched by the partitions. 
Figure \ref{f:mesh} depicts the sequence of the partitions generated 
by the algorithm \DISC   implemented within {\tt deal.II}.
\begin{figure}[ht!]
\centerline{
\includegraphics[width=0.3\textwidth]{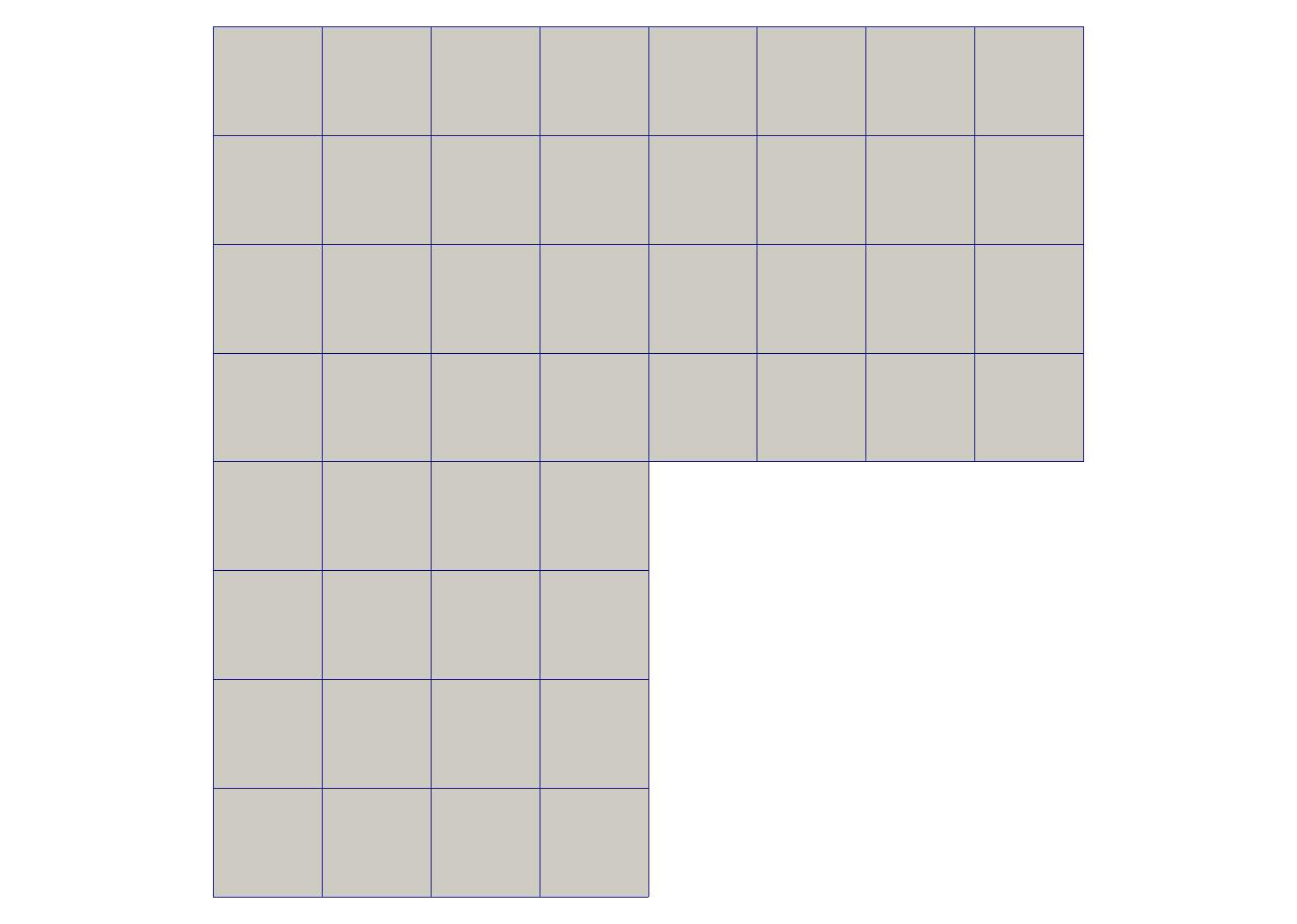}
\includegraphics[width=0.3\textwidth]{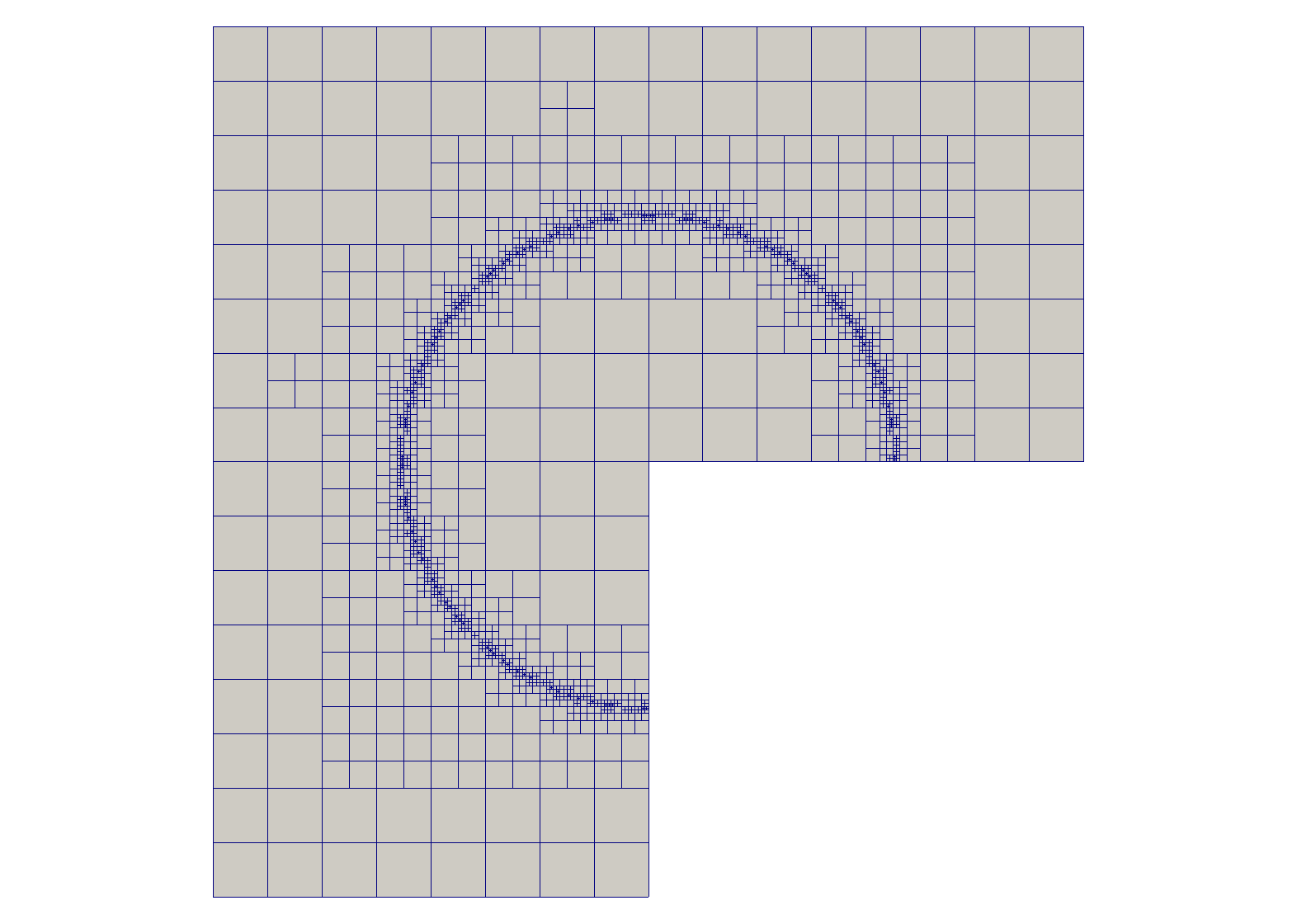}
\includegraphics[width=0.3\textwidth]{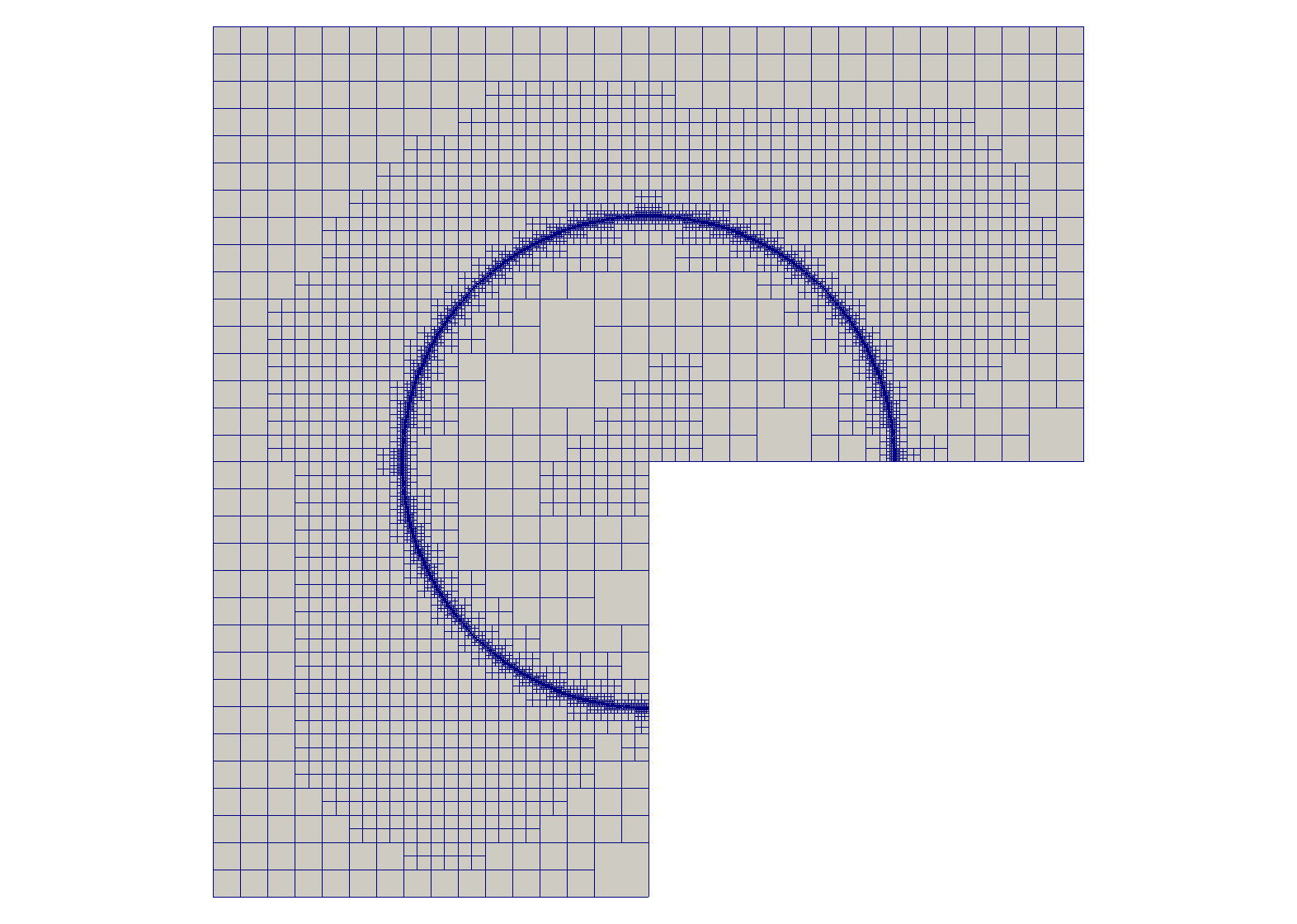}
}
\centerline{
\includegraphics[width=0.3\textwidth]{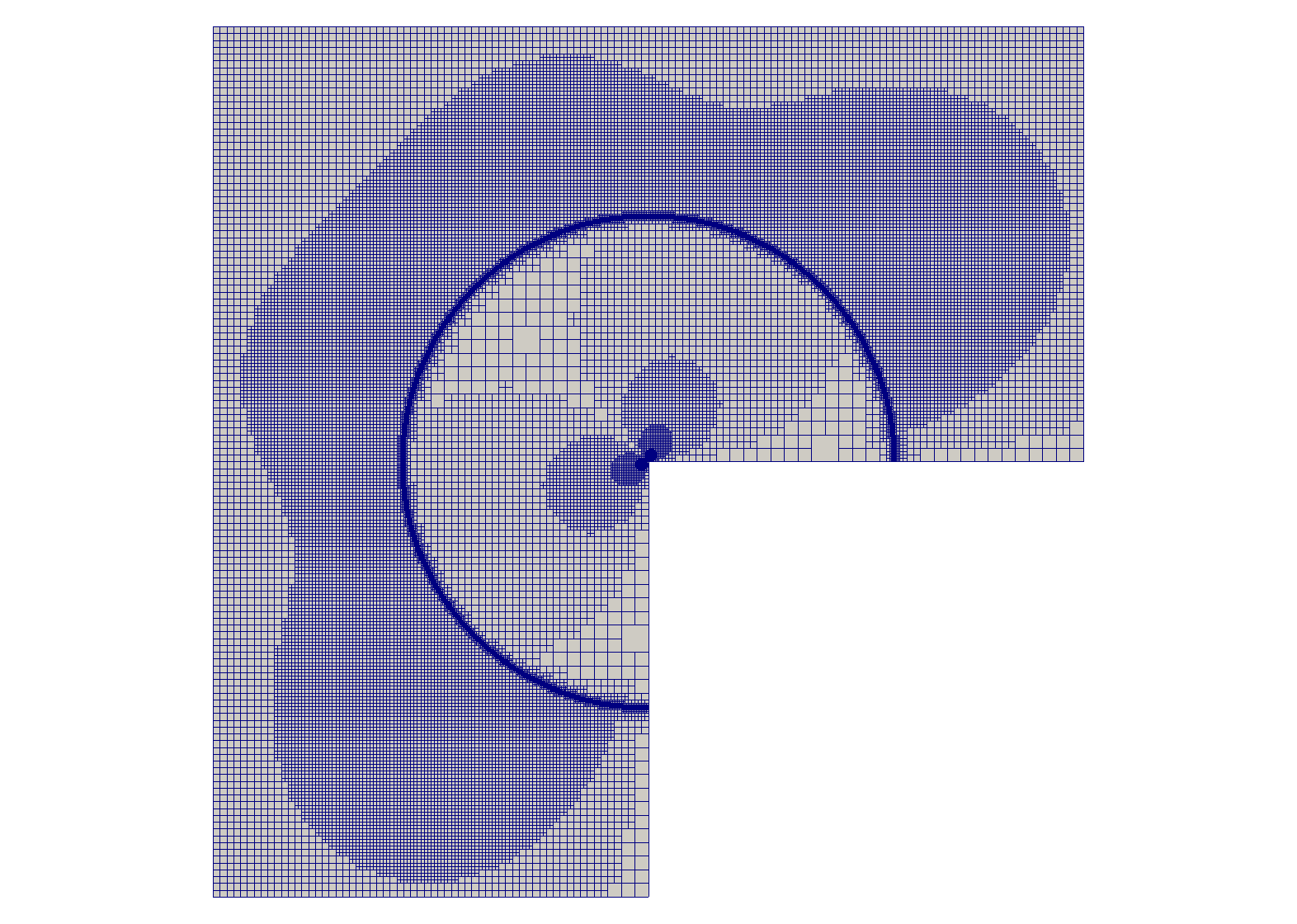}
\includegraphics[width=0.3\textwidth]{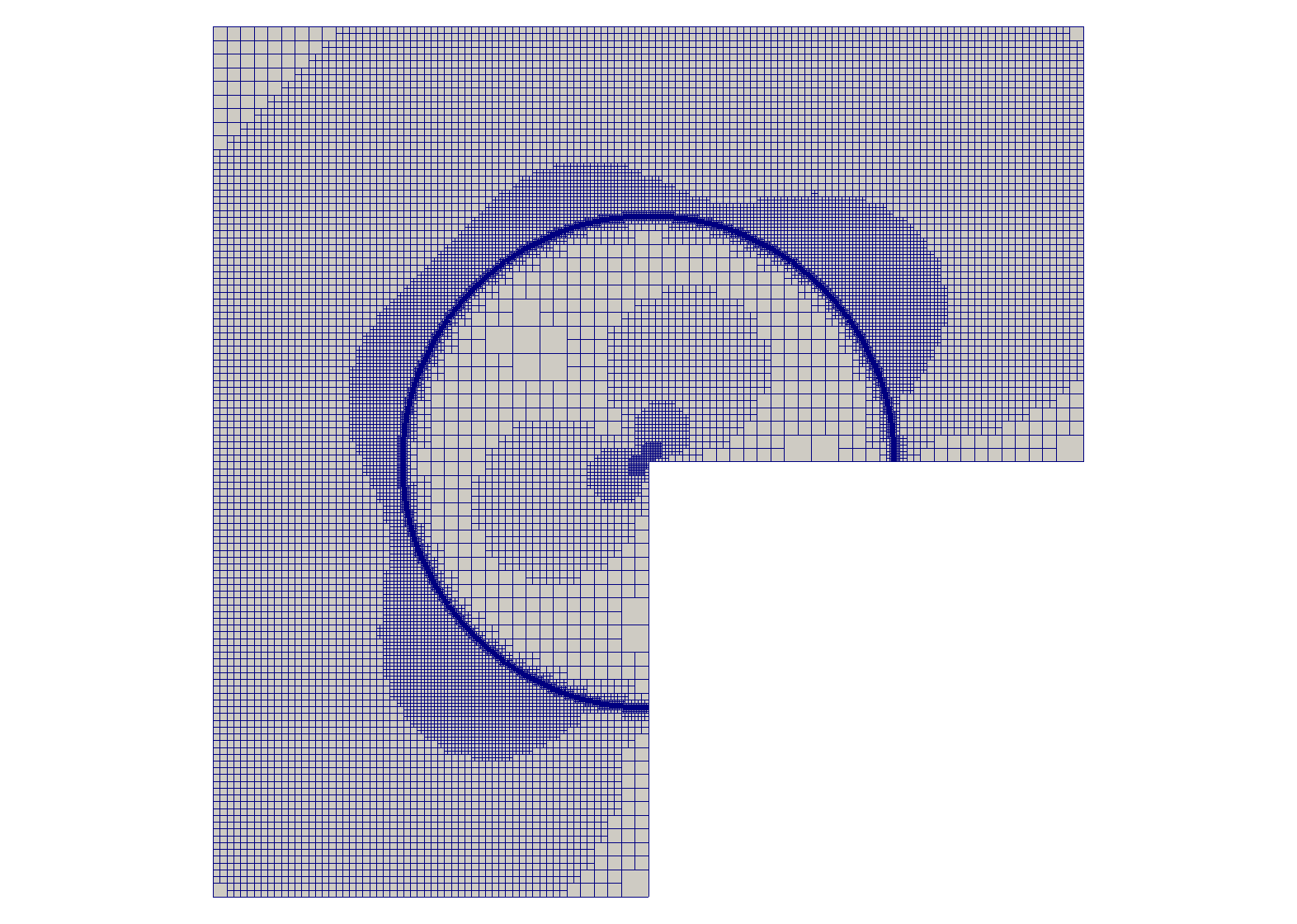}
\includegraphics[width=0.3\textwidth]{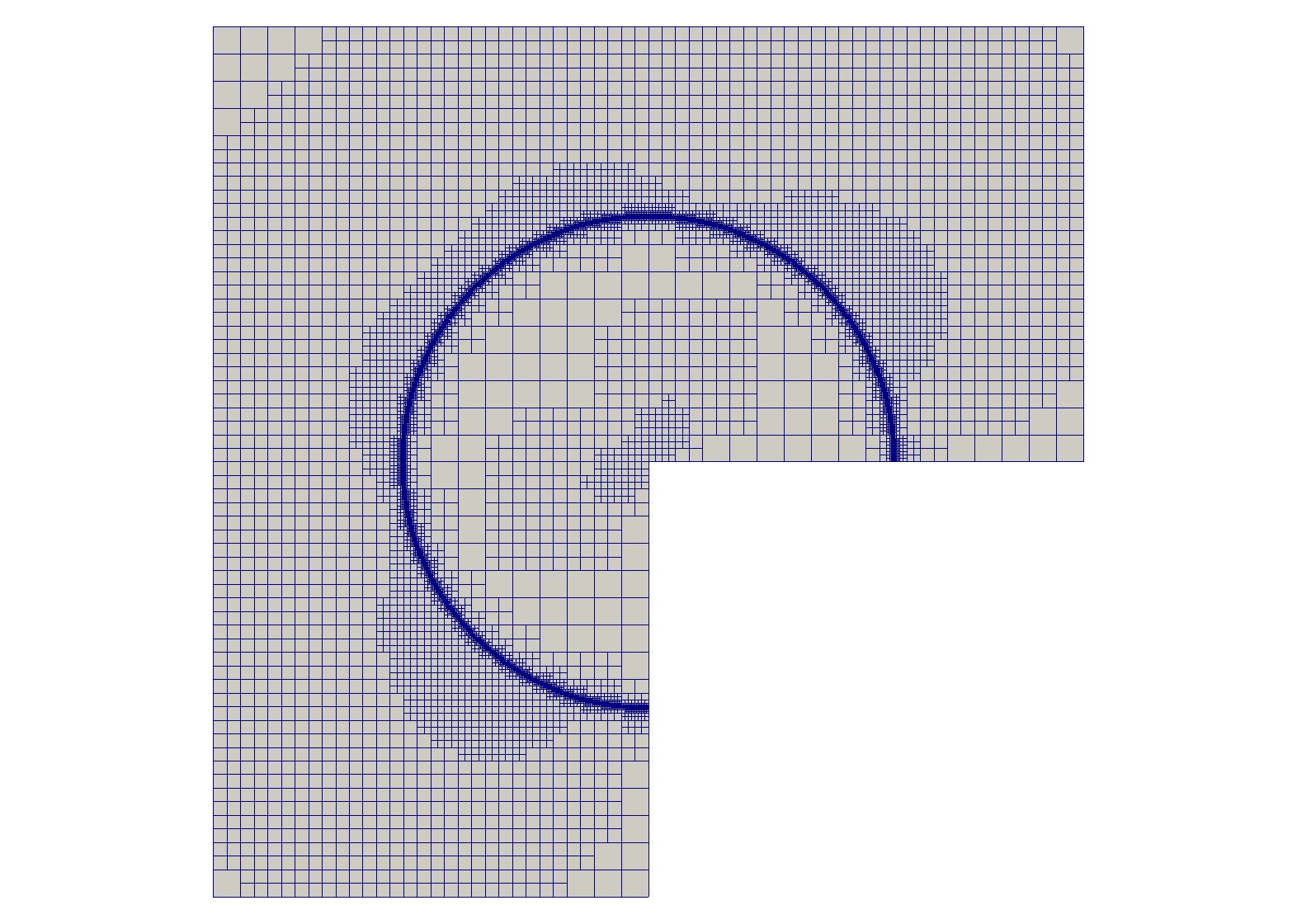}
}
\caption{\small{
Test 1 (L-shaped domain):
Sequence of partitions (clockwise) generated
  by the algorithm \bnew{$\DISC$} with  integrability index $q=2$ for $A$, 
and parameters $\beta=0.7,\omega=0.8,\veps_0=2$.} 
The initial partition (top left) is made of \bnew{uniform quadrilaterals
without} hanging nodes and 
all the \bnew{subsequent} partitions have at most one hanging node per side.
The algorithm \DISC refines at early stages only to capture the jump in the
diffusion. 
The refinements caused by the singular behavior of the solution 
 at the origin appear later in the adaptive process.}
\label{f:mesh}
\end{figure}

The parameters are chosen to be $\beta=0.7$, $\omega=0.8$ and $\veps_0=2$.
The standard AFEM loop in \MAIN is driven by error residual estimators 
together
with a D\"orfler \bnew{marking strategy  \cite{Dor96}} with parameter $\theta
= 0.3$, which is rather conservative.

We now discuss the choice of $p$ for which \CONDP is valid in view of 
Remark \ref{r:local}.
Let $\Omega=\Omega_1 \cup \Omega_2$ where $\Omega_1:= \{ (\rho,\delta)
\in \Omega : \rho \leq \rho_0/ 2 \}$
and $\Omega_2:= \Omega\backslash\Omega_1$.
The solution $u\in W^{1}(L_p(\Omega_1))$ for any $p<6$
in $\Omega_1$ and $u$ is Lipschitz in $\Omega_2$,
whence \CONDP is valid for $p<6$ in $\Omega_1$, i.e any
$q=2p/(p-2)>3$, and $p=\infty$ in $\Omega_2$, i.e. any $q\ge 2$.
Since the diffusion coefficient is constant on $\Omega_1$, it leads to
zero approximation error of $A$ and we only have to handle the 
jump of $A$ across the circular line $\{ \rho=\rho_0\}$ on $\Omega_2$.

Such a jump is never captured by the
partitions, thereby making $A$ never piecewise smooth over partitions 
of $\cT_0$ and preventing the use of a standard AFEM.
It is easy to check that for any $1\le q<\infty$, the matrix  $A$ is in $\cM^{1/q}(\mathcal T_0,L_q(\Omega))$.
Since the performance of $\DISC$ is reduced for larger $q$, according to \eqref{mt},
we should choose the smallest $q=2p/(p-1)$ compatible with $u\in
W^1(L_\infty(\Omega_2))$, namely $q=2$ for $p=\infty$.
 The right hand side $f$ satisfies $f \in
\cB^{1/2}(\cT_0,L_2(\Omega)) \subset \cB^{1/2}(\cT_0,H^{-1}(\Omega))$,
whereas the solution
$u \in \cA^{1/2}(\cT_0,H^1_0(\Omega))$ because $u\in \cA^{1/2}(\cT_0,H^1(\Omega_i))$,
$i=1,2$, and $\nabla u$ jumps over a Lipschitz curve \cite{CDN,DD:97}.

To test our theory, we take four different values of $p$ and thus 
the corresponding $q$ in our numerical experiments. For each of these 
different choices, Figure \ref{f:rate} (left) shows the decay of the  
energy error versus the number of degree of freedom in a $log-log$ scale. 
The experimental orders of convergence are
$$
-0.19   \quad\text{for} \quad q=6,
\qquad
-0.23  \quad\text{for} \quad q=5,
\qquad
-0.35  \quad\text{for} \quad q=3,
\qquad
-0.48   \quad\text{for} \quad q=2,
$$
in agreement with the approximability of $A$ stated above.
These computational rates are close to the expected values $-1/q$, and
reveal the importance of approximating and evaluating $A$ within
subdomains with the smallest Lebesgue exponent $q$ possible. In this
example $q=2$ yields an optimal rate of convergence for piecewise
bilinear elements.
Figure \ref{f:rate} (right) depicts the Galerkin solution after $6$ iterations
of $\DISC$ for $q=2$.
We finally point out that $\DISC$ with $q=\infty$,  
namely with $A$ being approximated in
$L_\infty(\Omega)$, cannot reduce the pointwise error in $A$ beyond
$3.96$ computationally which is consistent with the jump of $A$.
As a consequence,  any call of \COEFF with any smaller target
tolerance and $q=\infty$ does not converge.

\begin{figure}[ht!]
\centerline{\includegraphics[width=0.45\textwidth]{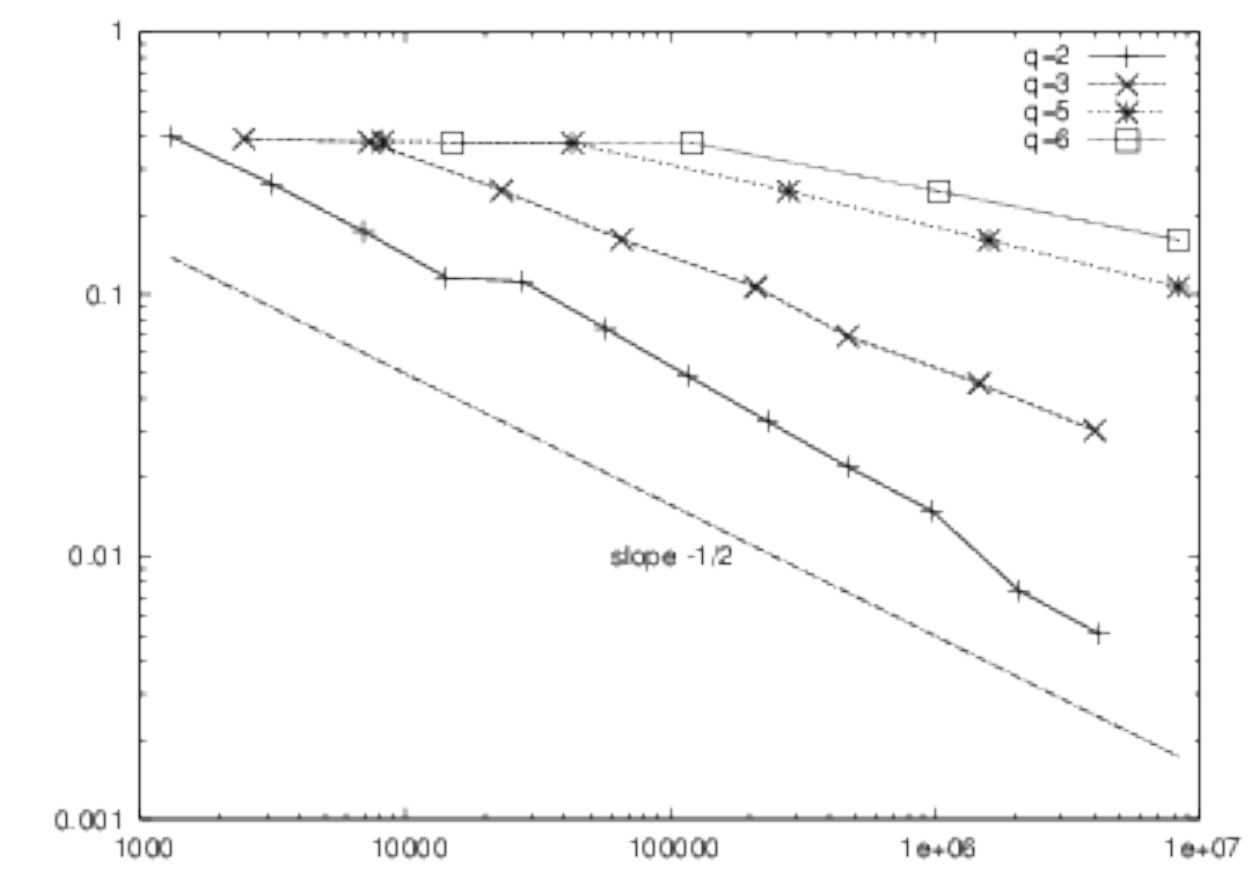}\includegraphics[width=0.45\textwidth]{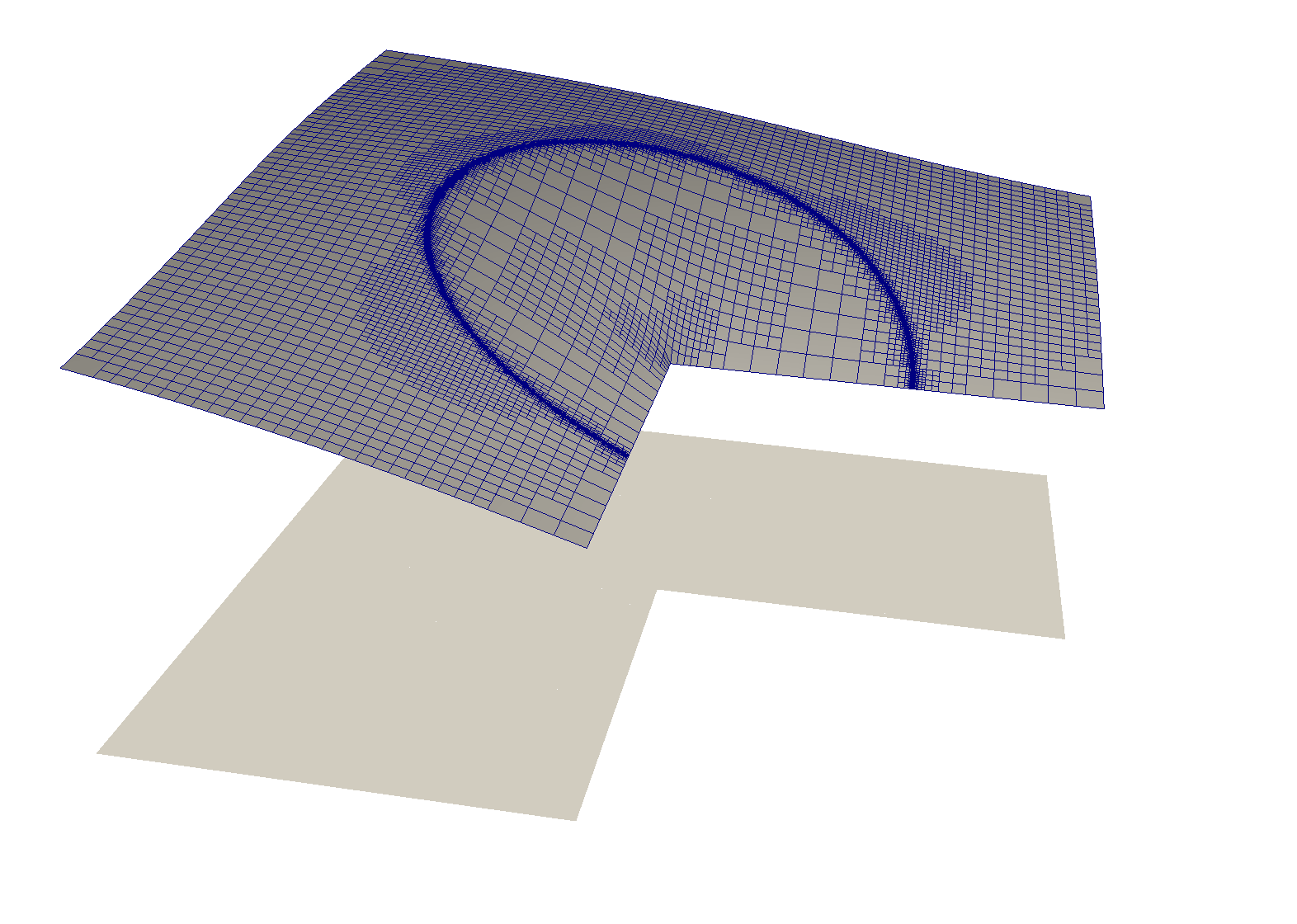}}
\caption{\small
Test 1 (L-shaped domain):
(Left) Energy error versus number of degrees of freedom for
values of $q=2,3,5,6$. The optimal rate of convergence is recovered for $q=2$.
(Right) The Galerkin solution together with the underlying partition
after $6$ iterations of the algorithm \DISC with $q=2$. 
The discontinuity of $A$  is never captured by the partitions and
the singularities of both $A$ and $\nabla u$ drive the refinements.} 
\label{f:rate}
\end{figure}

\subsection{Test 2: Checkerboard}\label{S:chekerboard}
We now examine $\DISC$ with an example which does not allow for $q=2$.
In this explicit example, originally suggested by Kellogg \cite{Kellogg41},
the line discontinuity of the diffusion matrix $A$ meets the
singularity of the solution $u$.
Let $\Omega = (-1,1)^2$, $A=a I$, where $I$ is the $2\times 2$
identity matrices and
$$
a(x,y) = \left\lbrace \begin{array}{ll}
b &\qquad  \mathrm{when} \quad (x-\frac{\sqrt{2}}{10})(y-\frac{\sqrt{2}}{10})\geq 0\\
1 &\qquad \mathrm{otherwise,}
\end{array}\right.
$$
with $b>0$ given.
The \bnew{forcing} is chosen to be $f\equiv 0$ so that with
appropriate boundary conditions, the solution $u$ in polar coordinates
$(\rho,\delta)$ centered at the point $(\frac{\sqrt{2}}{10},\frac{\sqrt{2}}{10})$ reads
$$
u(\rho,\delta)= \rho^\alpha \mu(\delta).
$$
where  $0<\alpha<2$ and
$$
\mu(\delta):= \left\lbrace 
\begin{array}{ll}
\cos((\frac \pi 2-\sigma)\alpha)\cos((\delta-\frac \pi 4)\alpha) & \qquad \mathrm{when} \quad 0\leq \delta <\frac \pi 2,\\
\cos(\frac \pi 4 \alpha )\cos((\delta-\pi+\sigma)\alpha) & \qquad \mathrm{when} \quad \frac \pi 2 \leq \delta <\pi,\\
\cos(\alpha \sigma)\cos((\delta- \frac{5\pi} 4)\alpha) & \qquad \mathrm{when} \quad \pi \leq \delta < \frac{3\pi} 2,\\
\cos(\frac{\pi} 4 \alpha)\cos((\delta-\frac{3\pi} 2-\sigma)\alpha) & \qquad \mathrm{when} \quad \frac{3\pi} 2\leq \delta < 2\pi. 
\end{array}
\right.
$$
The parameters $b$, $\alpha$ and $\sigma$ satisfy the non linear relations
\begin{equation*}
b = -\tan((\frac \pi 2-\sigma)\alpha)\cot(\frac \pi 4 \alpha ), \quad \frac 1 b  = -\tan(\frac \pi 4 \alpha)\cot(\sigma \alpha), \quad b = -\tan(\alpha \sigma)\cot(\frac \pi 4\alpha)
\end{equation*}
together with the constraints
$$
\max(0,\pi (\alpha-1)) < \frac \pi 2  \alpha  < \min(\pi \alpha, \pi), \qquad \max(0,\pi(1-\alpha))<-2\alpha \sigma < \min(\pi,\pi(2-\alpha)).
$$
We stress that the singular solution $u\in H^{1+s}(\Omega)$,
  $s<\alpha$, yet \bnew{$u\in\mathcal A^{1/2}(H^1_0(\Omega))$} \cite{NSV:09}. However,
the discontinuity of $A$ meets the singularity of $u$ and Remark
\ref{r:local} does no longer apply. In this case we have
$p<2/(1-\alpha)$ and $s_A=1/q=(p-2)/2p<\alpha/2$.

We challenge the algorithm \DISC with the approximate parameters
\begin{equation}\label{p:set1}
\alpha = 0.25, \quad b\approx 25.27414236908818, \quad \sigma \approx -5.49778714378214,
\end{equation}
which correspond to $p<8/3$ and $s_A<1/8$. We exploit \eqref{L2-best}
and choose $\bar A$ to be the meanvalue of $A$ element-by-element.
We report the experimental order of convergence (EOC)
of the energy error against the number of degrees
of freedom in Fig. \ref{f:kellogg} together with the solution at the final stage.
The asymptotic EOC (averaging the last 6 points)
is $-0.47$, which is about optimal and much better than the expected value
$s_A \approx -0.125$. On the other hand, the preasymptotic EOC
(without the last 6 points) is about $-0.2$. 
We will give a heuristic explanation of this
superconvergence rate in the following subsection.
We now conclude with Fig. \ref{f:kelloggs_mesh} which depicts
quadrilateral partitions at stages $k=0, 7, 22$.

\begin{figure}[ht!]
\centerline{\includegraphics[width=0.45\textwidth]{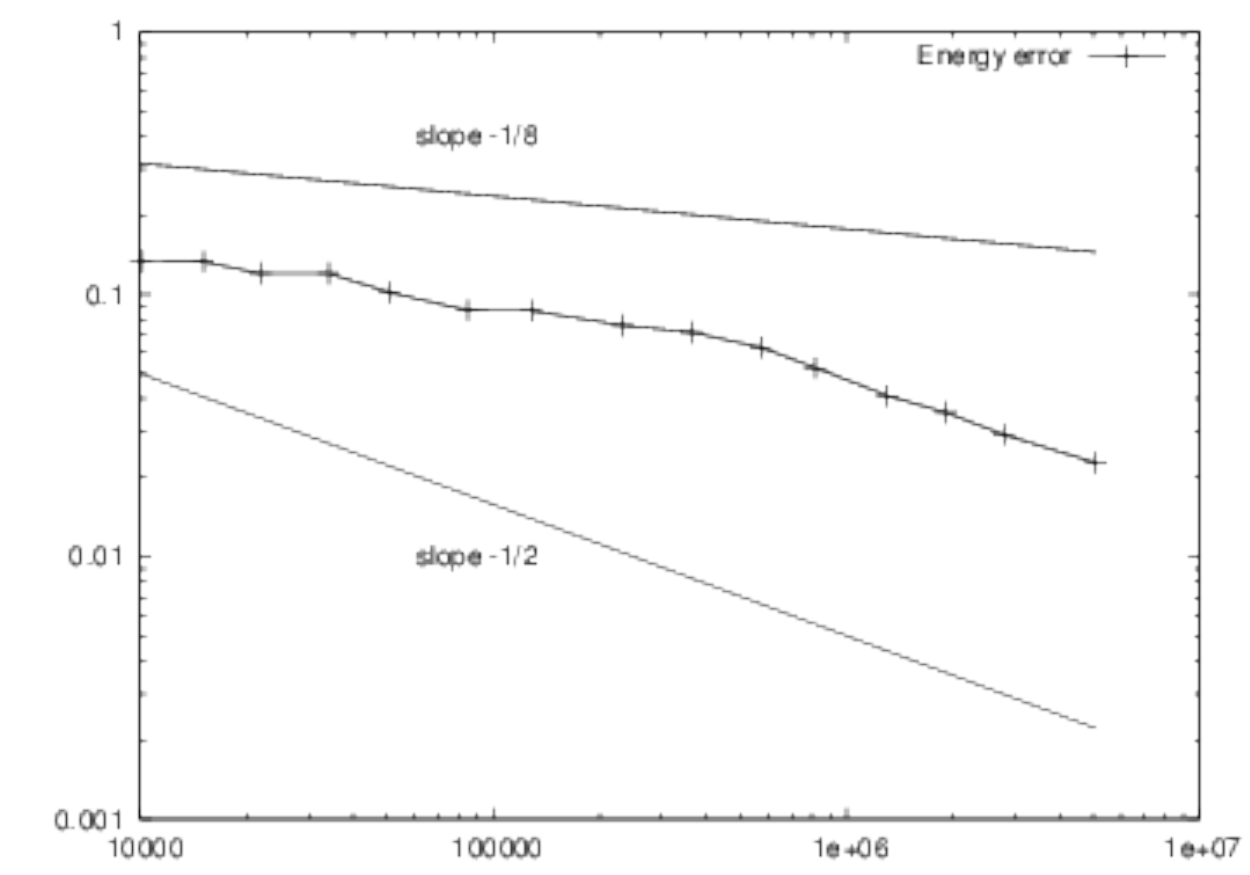}\includegraphics[width=0.45\textwidth]{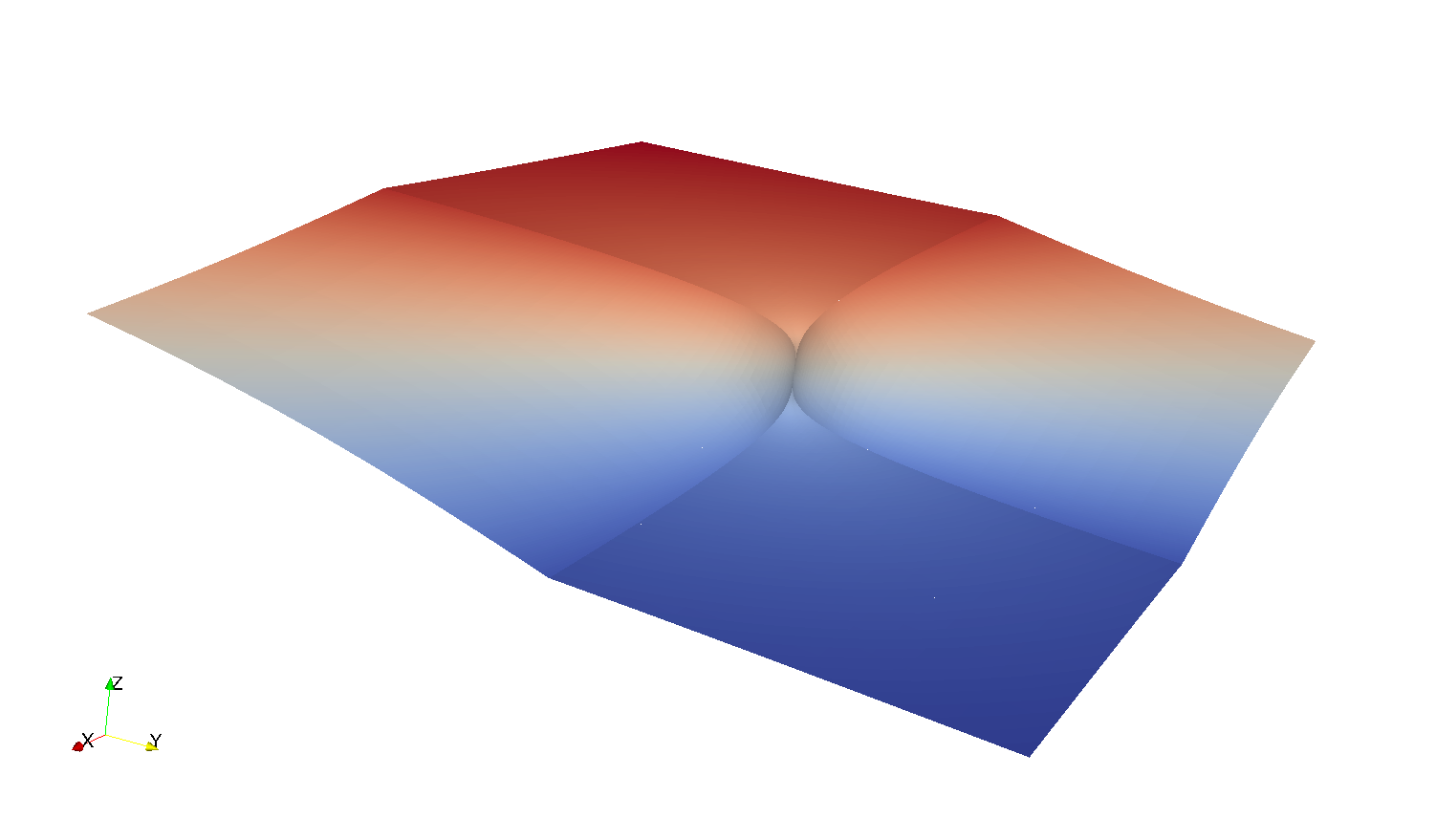}}
\caption{\small
Test 2 (Checkerboard): The parameters are chosen so that the solution $u\in H^{1+s}(\Omega)$, $s<0.25$. (Left) Energy error versus number of degrees of freedom.
The optimal rate of convergence $\approx -0.5$ is recovered.
(Right) The Galerkin solution together with the underlying partition
after $6$ iterations of the algorithm \DISC. 
The discontinuity of $A$  is never captured by the partitions and
the singularities of both $A$ and $\nabla u$ drive the refinements.
} 
\label{f:kellogg}
\end{figure}

\begin{figure}[ht!]
\centerline{\includegraphics[width=0.3\textwidth]{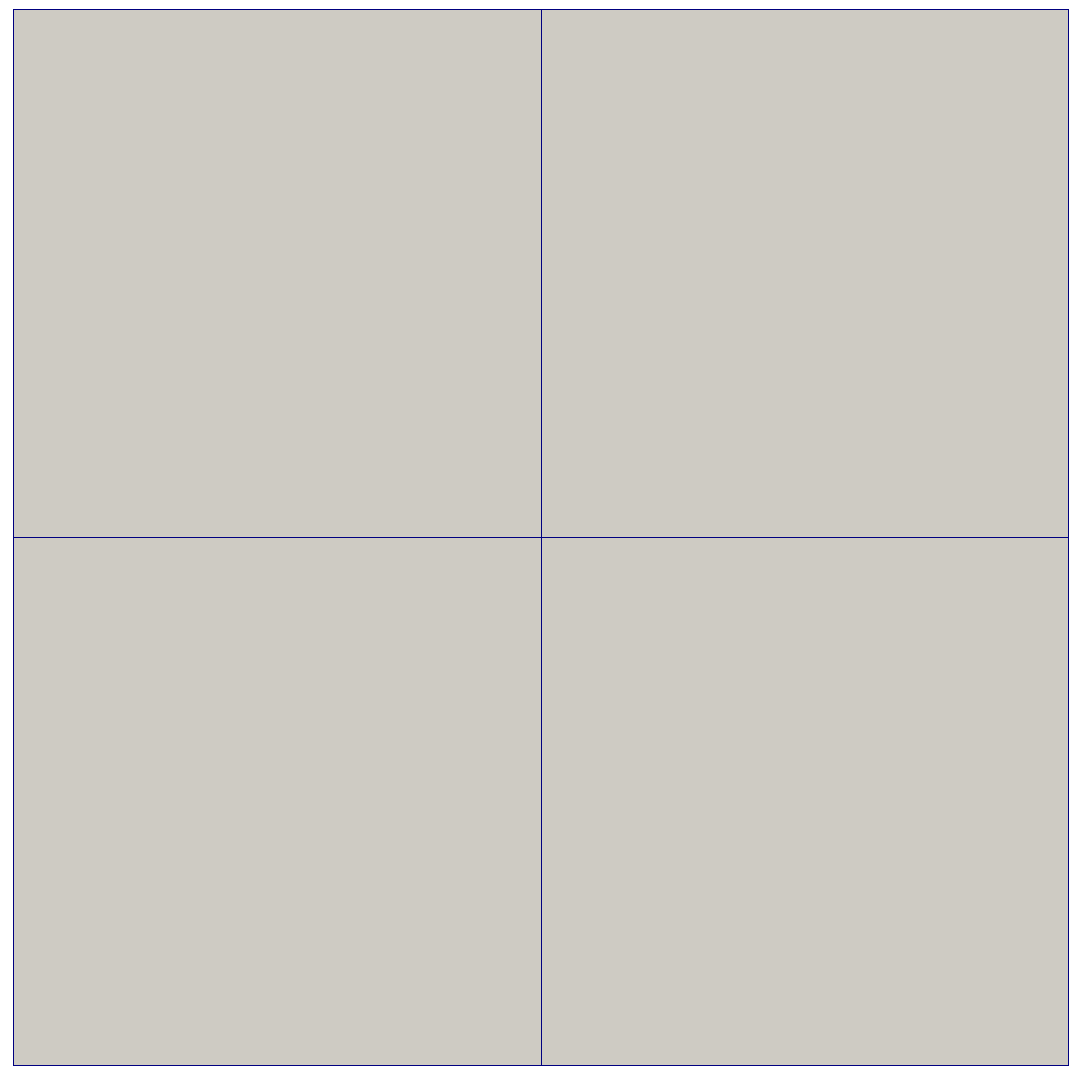}
\includegraphics[width=0.3\textwidth]{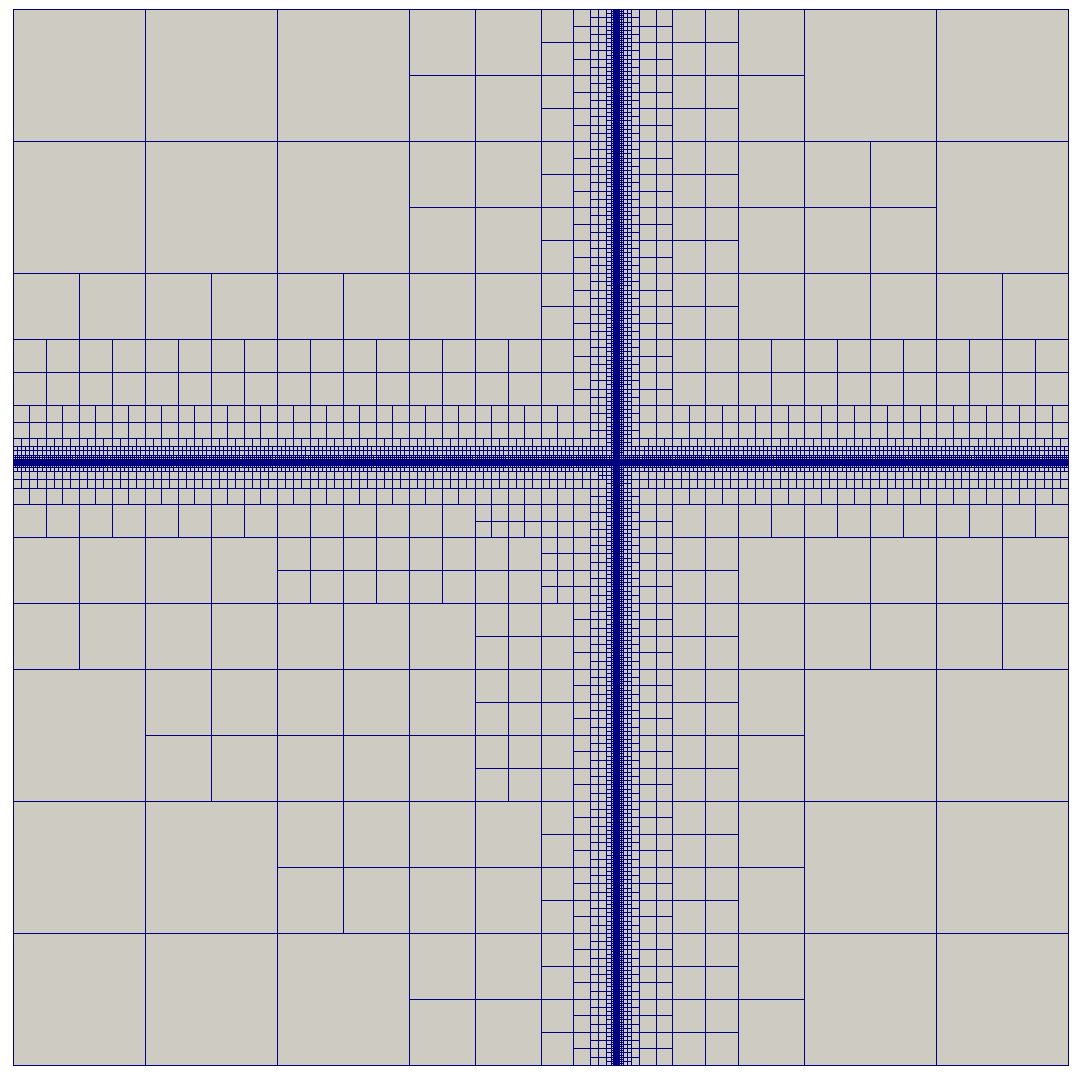}
\includegraphics[width=0.3\textwidth]{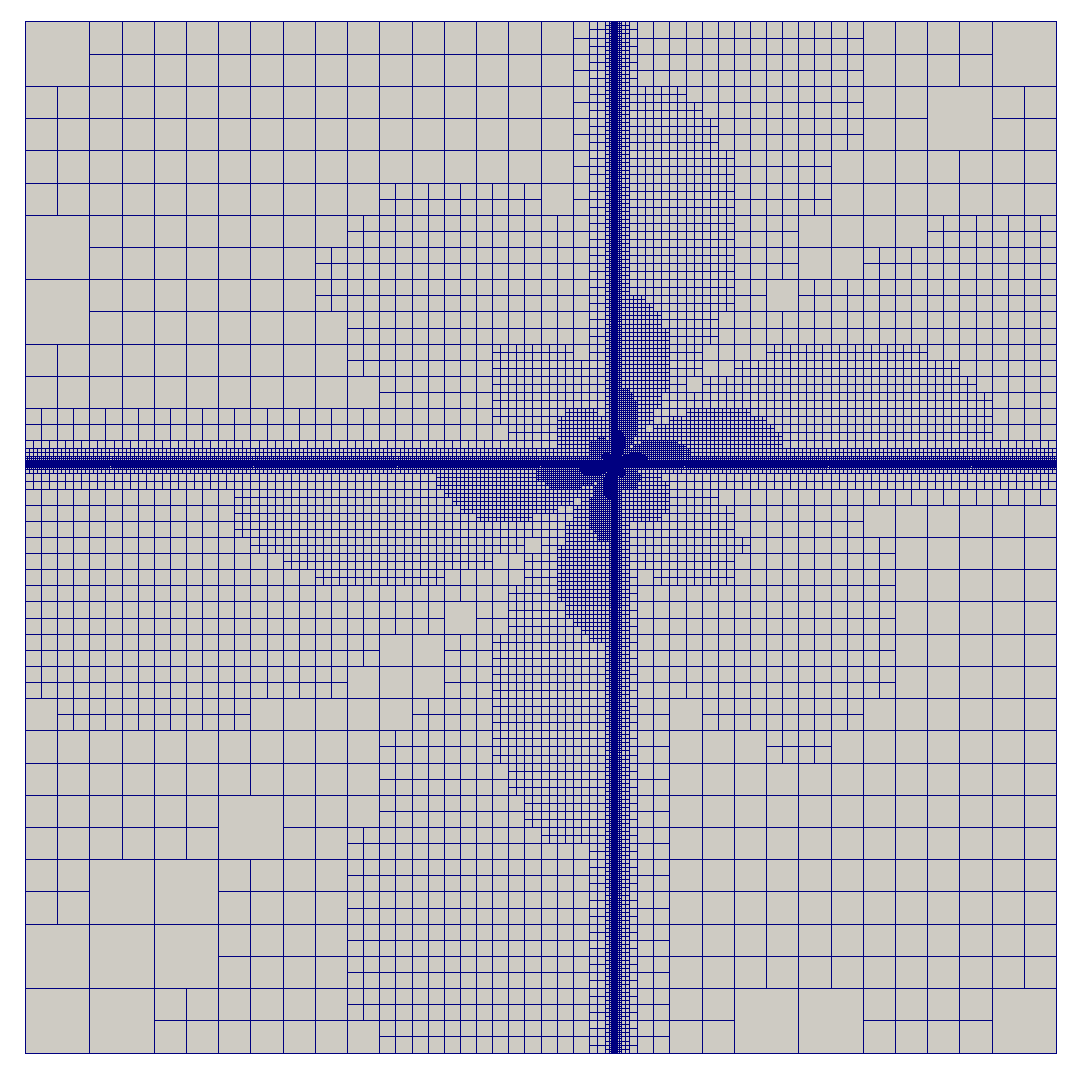}}
\caption{\small
Test 2 (Checkerboard):
Sequence of partitions (from left to right) generated by \DISC with
$\omega=0.8$.
The initial partition (first) is made of four quadrilaterals, 
The algorithm refines at early stages only to capture the discontinuity in the diffusion coefficient (second).
Later the singularity of $u$ comes into play and, together with \bnew{that of
$A$}, drives the refinement (third).
The corresponding subdivision consists of $5$ million degrees of freedom.
The smallest cell has a diameter of $2^{-8}$ which illustrates
the strongly graded mesh constructed by \DISC.
} 
\label{f:kelloggs_mesh}
\end{figure}

\subsection{Performance of $\DISC$ with Interacting Jump and Corner Singularities}

We finally give a heuristic explanation to the surprising
superconvergence behavior of $\DISC$ in Test 2.
Let $u\approx \rho^\alpha$, with $0<\alpha<1$, be the prototype solution 
such as that of Section \ref{S:chekerboard}. Let $A$ be a discontinuous 
diffusion matrix with discontinuity across a Lipschitz curve $\Gamma$ emanating
from the origin, and let $\bar A$ be its \bnew{local} meanvalue.

Let $\omega_j$ be the annulus $\{x\in\Omega: 2^{-(j+1)} < \rho=|x| < 2^{-j}\}$ for
$0\le j \le J$ and set $\omega_{J+1}:= \{x\in\Omega: |x| < 2^{-(J+1)}\}$.
We assume that $c^{-1} h_j^2\leq  |T| \leq c h_j^2$ for each element $T$ within $\omega_j$ touching $\Gamma$, where $c$ is a constant independent of $j$ and the total number of cells $N$.
\bnew{Revisiting the proof of the perturbation theorem (Theorem
  \ref{T:perturbation}), we realize that the error \bnew{$E_A$} 
due to the approximation of $A$ can be decomposed as follows:}
$$
E_A^2:= \| (A-\bar A)\nabla u\|_{L_2(\Omega)}^2 \bnew{\approx} \sum_{j=0}^{J+1} \|
(A-\bar A)\nabla u\|_{L_2(\omega_j)}^2 = \sum_{j=0}^{J+1} \delta_j^2,
\qquad \delta_j :=  \| (A-\bar A)\nabla u\|_{L_2(\omega_j)}.
$$
We choose $q=2$ and $p=\infty$ away from the origin which implies
that the contribution $\delta_j$ within $\omega_j$ is estimated by
\begin{equation*}
\delta_j \leq \|A-\bar A\|_{L_2(\omega_j)} \|\nabla u\|_{L_\infty(\omega_j)}, \quad 0\leq j \leq J.
\end{equation*}
The first term is simply the square root of the area around the interface and within
$\omega_j$, which amounts to
$
\|A-\bar A\|_{L_2(\omega_j)} \approx \big( h_j 2^{-j} \big)^{1/2}
\approx \big( N_j^{-1} 2^{-2j} \big)^{1/2},
$
with  $N_j\approx h_j^{-1} 2^{-j}$ being the number of elements touching $\Gamma$ within $\omega_j$.
The second term reduces to
$
\|\nabla u\|_{L^\infty(\omega_j)} \approx 2^{-j(\alpha-1)}, 
$
whence
\begin{equation*}
\delta_j \approx N_j^{-1/2} 2^{-j\alpha}.
\end{equation*}
We further assume error equidistribution, which
entails $\delta_j^2 \approx \Lambda$ constant independent of $j$. This
implies
$$
N_j \approx \Lambda^{-1} 2^{-2j\alpha}
\qquad \Rightarrow \qquad
N \approx \sum_{j=0}^J N_j  \approx \Lambda^{-1} \sum_{j=0}^J 2^{-2j\alpha} \approx  \Lambda^{-1} 
$$
because $\alpha>0$. 
It remains to determine the value of $J$. 
On $\omega_{J+1}$ we have $\nabla u \in L_p$, $p<\frac{2}{1-\alpha}$, so that with $q=\frac{2p}{p-2}$, the contribution from $\omega_{J+1}$ is estimated by
$$
\delta_{J+1}^2 \leq \| A-\bar A \|^2_{L_q(\omega_{J+1})} \| \nabla
u\|^2_{L_p(\omega_{J+1})}  \preceq | \omega_{J+1}|^{2/q}
\preceq 2^{-\frac{4J}{q}}
$$
with a hidden constant that blows up as $q$ approaches the limiting
value $2/\alpha$.
Matching the error $\delta_{J+1}^2$ 
with $\Lambda$ gives rise to the relation
$$
N \approx 2^{4J/q}
\qquad\Rightarrow\qquad
J \approx \log N.
$$
We thus conclude that
$$
E_A \approx  \big(J/N\big)^{1/2} \approx N^{-1/2} |\log N|^{1/2}.
$$

The ensuing mesh has a graded meshsize $h_j\approx N^{-1}
2^{-j(1-2\alpha)}$ towards the origin provided $\alpha<1/2$; if
$\alpha\ge 1/2$ then uniform refinement suffices. Such a graded mesh cannot
result from the application of $\COEFF$ because it only measures the jump
discontinuity of $A$ which is independent of the distance to the
origin. 

However, the refinement due to $\MAIN$ could be much more
severe because the best bilinear approximation $\hat U$ of $u$ on $T \subset \omega_j$ reads
$$
E_u(T)^2:= \| u- \hat U\|_{H^1(T)}^2  \approx h_j^2 2^{2j(2-\alpha)}
$$
and equidistribution $E_u(T)\approx\lambda$ (constant) yields a graded
meshsize $h_j \approx \lambda 2^{-j(2-\alpha)}$. This grading is
stronger than that due to $A$ and asymptotically dominates. This in
turn explains the preasymptotic EOC of Fig. \ref{f:kellogg} and the
quasi-optimal asymptotic EOC also of Fig. \ref{f:kellogg}.

\end{document}